\documentclass[12pt, a4paper]{amsart}
\usepackage{amsmath,amsthm}
\usepackage{verbatim}
\usepackage{marvosym}
\usepackage{graphicx}
\usepackage{color}
\usepackage{epsfig}
\usepackage{ulem}
\usepackage{rotating}
\usepackage{lscape}
\usepackage{hyperref}
\usepackage[all]{xy}
\usepackage{soul}
\usepackage{graphics}
\usepackage{epstopdf}
\usepackage{tikz}
\usepackage{float}
\usepackage[section]{algorithm}
\usepackage{algorithmic}

\usepackage{caption}
\usepackage[labelformat=simple]{subcaption}

\usepackage{geometry}
\theoremstyle{plain}
\newtheorem{thm}{Theorem}[section]
\newtheorem{cor}[thm]{Corollary}

\newtheorem{prop}[thm]{Proposition}
\newtheorem{lemma}[thm]{Lemma}

\theoremstyle{definition}
\newtheorem{definition}[thm]{Definition}
\newtheorem{remark}[thm]{Remark}

\newtheorem{example}[thm]{Example}
\newtheorem{alg}[thm]{Algorithm}

\makeatletter
\newtheorem {rep@theorem}{\rep@title}
\newcommand{\newreptheorem}[2]{%
\newenvironment{rep#1}[1]{%
 \def\rep@title{#2 \ref{##1}}%
 \begin{rep@theorem}}%
 {\end{rep@theorem}}}
\makeatother

\newreptheorem{theorem}{Theorem}

 \newcommand{\Ax}{\ensuremath{{\mathcal{A}}}}
 \newcommand{\Rx}{\ensuremath{{\mathcal{R}}}}
 \newcommand{\Cx}{\ensuremath{{\mathcal{C}}}}
 \newcommand{\Lx}{\ensuremath{{\mathcal{L}}}}
 \newcommand{\Mx}{\ensuremath{{\mathcal{M}}}}
  
 \newcommand{\Px}{\ensuremath{{\mathcal{P}}}}
 \newcommand{\Wx}{\ensuremath{{\mathcal{W}}}}

\newcommand{\Q}{\ensuremath{\mathbb{Q}}}

\newcommand{\Z}{\ensuremath{\mathbb{Z}}}
\newcommand{\C}{\ensuremath{\mathbb{C}}}

\renewcommand{\H}{\ensuremath{\mathbb{H}}}

\newcommand{\SL}{\text{SL}}
\newcommand{\PSL}{\text{PSL}}

\newcommand{\tr}{\text{tr}}

\newcounter{nootje}
\newcommand{\mat}[4]{\left(\begin{array}{cc}#1 & #2 \\ #3 & #4 \end{array}\right)}

\numberwithin{equation}{theorem}

\begin{document}
\title{Geometric structures and $PSL_2(\C)$ representations of knot groups from knot diagrams}
\author{Kathleen L. Petersen and Anastasiia Tsvietkova}

\begin{abstract} We describe a new method of producing equations for the canonical component of representation variety of a knot group into $\rm{PSL}_2(\C)$. Unlike known methods, this one does not involve any polyhedral decomposition or triangulation of the knot complement, and uses only a knot diagram satisfying a few mild restrictions. This gives a simple algorithm that can often be performed by hand, and in many cases, for an infinite family of knots at once. The algorithm yields an explicit description for the hyperbolic structures (complete or incomplete) that correspond to geometric representations of a hyperbolic knot. As an illustration, we give the formulas for the equations for the variety of closed alternating braids $(\sigma_1(\sigma_2)^{-1})^n$ that depend only on $n$.

\end{abstract}

\maketitle

\section{Introduction}

Let $K$ be a knot in $S^3$ admitting a diagram $D$ satisfying a few mild restrictions, and let $M=S^3-N(K)$ be the complement of a tubular neighborhood of the knot. We present a new algorithm that produces equations for all geometric representations of the fundamental group $\pi_1(M)\rightarrow\text{PSL}_2(\C)$ based solely on $D$. We choose a preferred conjugate for a meridian, and therefore our algorithm gives equations for a reduced representation variety for the knot complement.
We also give presentations for Wirtinger generators. Hence, this is a new method to efficiently compute the canonical component of $\text{PSL}_2(\C)$-character variety of $M$. Due to the well understood lifting of representations from $\PSL_2(\C)$ to $\SL_2(\C)$ our method also computes  $\SL_2(\C)$ representations. This is inspired by the work of Thistlethwaite and Tsvietkova \cite{MR3190595, thesis}, who developed a similar algorithm for determining parabolic representations (conjecturally, all of them) of links once a suitable diagram is given.

 The study of representations of knot groups into $\mathrm{PSL}_2(\C)$  has a long history: we refer the reader to \cite{MR0222880},  \cite{MR0300267}, \cite {MR0413078}, and \cite{MR745421}.
The character variety, which is the set of all representations up to trace equivalence, has many broad applications to 3-manifold topology as, for example, in \cite{MR1670053}, \cite{MR1288467}.
Notably,  Culler and Shalen \cite{MR683804} showed how to associate essential surfaces in a 3-manifold to ideal points of its character variety.
Character varieties were also fundamental tools in the proofs of the cyclic \cite{MR881270} and finite \cite{MR1909249} surgery theorems.

The set of all representations of $\pi_1(M)$ into $\rm{PSL}_2(\C)$ is a complex algebraic set.  Any component containing a discrete and faithful representation is called a canonical component. We use the term geometric representation to denote any representation whose developing map has a particularly nice geometric format, as outlined by Thurston in \cite{thurston}.  

Verifying that our algorithm gives a representation is straightforward. It is more difficult to prove that all geometric representations (and there are infinitely many of them) can be obtained this way, which we establish in our main theorem.

Representation and character varieties for specific knots are often computed in an ad hoc manner, usually for knot complements whose fundamental group admits a particularly nice presentation. They have been computed for a few infinite families of such knots, for example, in \cite{MR2827003, MR3073918, MR3450771}. 

There are also more general  approaches, albeit previously these were either limited to parabolic representations, or relied on the use of software in practice. In particular, for parabolic representations, in addition to Thistlethwaite and Tsvietkova's work \cite{MR3190595}, another approach, with some similarities, was recently developed in \cite{MR3876301}. Beyond the parabolic case, the known algorithm starts with a suitable triangulation (see \cite{MR3325748}). But in general, there is no algorithm that provides such a triangulation for a 3-manifold, and only a procedure exists as a part of the program SnapPea \cite{SnapPea}. The procedure is based on heuristically retriangulating the 3-manifold until a desired triangulation is obtained. For orientable irreducible 3-manifolds with one cusp, where the 3-manifold is small (i.e. every embedded closed incompressible surface is boundary parallel), an algorithm for obtaining a triangulation that will yield a generalized variety was given by Segerman \cite{Segerman}. Note that many knot complements in 3-sphere are not small.

Our algorithm uses only a taut diagram of the knot, and does not use any polyhedral decomposition of the knot complement. Conjecturally, every link has a taut diagram. Additionally, many diagrams are known to be taut: e.g. reduced alternating diagrams of hyperbolic alternating links \cite{MR3190595}, and some other infinite families (this is discussed in detail in subsection \ref{suitable}). Note that it is not known, for example, how to find a suitable triangulation for a hyperbolic alternating knot algorithmically. Therefore, we give the first algorithm for computing the equations for character variety of these wide classes of knots. If the input diagram is not taut, our algorithm may fail to give representations, but representations it produces will be valid.

Our approach was implemented in software \cite{Software}. However in many cases, it allows to produce formulas for varieties of an infinite family of knots just by looking at their diagrams, by hand, as we show in the last section. 

The way we start is similar to that of Thistlethwaite and Tsvietkova \cite{MR3190595}: by taking certain arcs in a knot complement, and considering the respective isometries in the covering space, $\mathbb{H}^3$. But since we do not limit this to the parabolic case anymore, the geometric picture for preimages of the arcs in $\mathbb{H}^3$ is not taking place in the familiar ``horoball" structure anymore. Instead, the preimages of the arcs may correspond to  parabolic, loxodromic, or elliptic transformations in $\mathbb{H}^3$ between or along the shapes that we call bananas. This, in its turn, requires more care when we work with the respective elements of $\mathrm{PSL}_2(\C)$: we prove a number of lemmas about such arcs and elements, allowing us to work with well-defined and ``straightened" geodesics in incomplete hyperbolic structures.

Our argument uses the fact that all meridianal curves are homotopic, and so doesn't immediately extend to links.  A generalization to links is possible, but has some technical challenges, which are discussed in Remark \ref{links}.

\subsection{Notation} For the benefit of the reader, we  collect some notation that is frequently used.  These terms are defined as needed, later in the text. 

We will use the upper half space model of $\H^3$, with the Riemannian sphere $\C\cup\{\infty\}$ serving as the boundary of the hyperbolic space. We reserve the following symbols: $K$ is a knot in $S^3$, $N(K)$ is a tubular neighborhood of $K$,  $M=S^3-N(K)$ is its complement in 3-sphere,  $T=\partial M$, $D$ a diagram for $K$ (i.e. a projection of $K$ to a 2-sphere), $B$ a lift of cusp cross-section (of $N(K)$) to $\H^3$ (later we will call this a `banana'), $H$ a cover of $M$ in $\H^3$.  The symbols $\mu$ and $\lambda$ will denote a standard meridian and longitude, based at a point $b$ on $T$.  We use $\gamma$ to denote a crossing arc, and $\beta$ to denote a peripheral arc, and $\alpha$ to denote a path along the top of the knot.  We also use $\tilde{a}$ for a preimage of an arc $a$ in $H$.   We use $\rho$ to denote a representation from $\pi_1(M)$ to $\text{PSL}_2(\C)$. We will use $I$ to denote the coset containing the $2\times 2$ identity matrix in $\PSL_2(\C)$.  Except when signs are needed for clarity, we use matrices for elements of $\PSL_2(\C)$.

\subsection{Outline} If $K$ is a hyperbolic knot, then there are infinitely many geometric representations of $\pi_1(K)$ in any neighborhood of a discrete and faithful representation.  These representations correspond to covers of $S^3-N(K)$ by particularly nice subsets of $\H^3$, or, in other words, to geometric structures on the knot complement.  Given a hyperbolic structure on $S^3-N(K)$, using the developing map, a loop in the fundamental group of $S^3-K$ lifts to (infinitely many) paths in $\H^3$.  A geodesic in $M$ lifts to geodesic arcs in $\H^3$. Once one chooses a preferred lift, such an arc corresponds to a unique isometry. This isometry fixes the geodesic as a set in $\H^3$ and sends a lift of the base point to the next lift of this point along the geodesic.  The isometry can be identified with a unique element of $\PSL_2(\C)$.   We will use this correspondence between the  arcs and paths in $S^3-N(K)$ (or, alternatively, in a knot diagram) and the elements of $\text{PSL}_2(\C)$ throughout the paper.

 Section \ref{Background} recalls some basics about $\SL_2(\C)$ and $\PSL_2(\C)$-character varieties of knot complements in 3-sphere, and the respective geometric structures.
In Section~\ref{GeometricSetup} we first discuss knot diagrams suitable for our algorithm. We then proceed to define certain types of arcs in $S^3-N(K)$, visible in a knot diagram, and establish a correspondence between the arcs and elements of $\PSL_2(\C)$. The arcs are meridian, peripheral and crossing arcs. The arcs can be concatenated into paths and loops that correspond to Wirtinger generators. We show that geometric representations can be extended to such arcs, paths and loops in Section \ref{Paths}. In Section \ref{Normalizing},  we prove that the elements that correspond to arcs and paths in non-degenerate geometries are conjugate to a particularly nice (``normalized") elements of $\PSL_2(\C)$. We also write down relations for these elements of $\PSL_2(\C)$ from a knot diagram.  Finally in Section~\ref{section:wirtinger}, we prove that not only does our set-up determine geometric representations, but all the representations of the knot group into $\PSL_2(\C)$ that lie on the canonical component of the variety are determined in this way (Theorem \ref{ThmMain1} and Corollary \ref{corollary:geometric}). Section \ref{section:algorithm} is devoted to the algorithm for writing out the equations for the geometric component of the character variety from the knot diagram, and to some simple practical shortcuts. We prove that algorithm works in Theorem \ref{ThmAlg}. Section \ref{section:smallregions} is devoted to less trivial shortcuts: simplification of the algorithm for bigons and 3-sided regions of a knot diagram. Section \ref{CuspShape} is a short comment on computing cusp shape for different geometric structures. Section \ref{figure-eight} is an example that illustrates the algorithm: we take figure-eight knot, and give a simple computation of the equations for the canonical component of $\PSL_2(\C)$ and $\SL_2(\C)$- character varieties, equations for parabolic representations, for traceless representations, for cusp shapes for different geometric structures, and for the A-polynomial.  In Section \ref{3-braids}, we show how algorithm can be applied to an infinite family of knots at once: for this, we choose the knots that are closed 3-braids with braid word $(\sigma_1 (\sigma_2)^{-1})^n$, and obtain formulas for the variety of such a knot that depend only on $n$.  Magnus \cite{MR387430} computed representations for this family.

\section*{Acknowlegements}

The authors thank Marc Culler for helpful conversations.  The first author was partially supported by an AMS-Simons Research Enhancement Grant for PUI Faculty. The second author was partially supported by NSF CAREER grant DMS-2142487, by NSF research grants DMS-2005496, DMS-1664425, DMS-1406588, by Institute of Advanced Study under DMS-1926686 grant, and by Okinawa Institute of Science and Technology.  

\section{Background: Geometric Representations and Characters}\label{Background}

In this section, we review some of the well-known facts about representation varieties and geometric structures of a 3-manifold.

Let $\Gamma$ be a finitely presented group.  The $\SL_2(\C)$ {\em representation variety} $R(\Gamma)$ is the set of all representations from $\Gamma$ to $\SL_2(\C)$,
\[ R(\Gamma)=\{\rho: \Gamma \rightarrow \SL_2(\C)\}.\]
As conjugate representations correspond to the same geometric structure, the $\SL_2(\C)$ character variety is often useful.  It is
\[ X(\Gamma) = \{ \chi_{\rho} : \  \rho\in R(\Gamma) \} \]
where the character function $\chi_{\rho}:\Gamma\rightarrow \C$ is defined by $\chi_{\rho}(\gamma) = \text{trace}(\rho(\gamma)).$
A representation is called {\em irreducible} if it is not conjugate to an upper triangular representation. If $\rho$ is an irreducible representation, then $\chi_{\rho}=\chi_{\rho'}$ exactly when $\rho$ and $\rho'$ are conjugate.

Both $R(\Gamma)$ and $X(\Gamma)$ are affine, complex algebraic sets defined over $\Q$. Different presentations for a group yield isomorphic sets.  Therefore, for a 3-manifold $M$ we often write $X(M)$, for example,  to mean $X(\pi_1(M))$ up to isomorphism.   When $M=S^3-N(K)$, we outline a construction of the $\PSL_2(\C)$ character variety below.  We denote this set by $Y(M)$.

If $M$ is hyperbolic, any component of $X(M)$ or $Y(M)$   that contains the character of a discrete and faithful representation is called a {\em canonical component.}  Thurston \cite{thurston} showed that for the fundamental group of a hyperbolic manifold, the complex dimension of a  canonical component equals the number of cusps of the manifold.  Therefore, if $M$ is a hyperbolic knot complement, then  the canonical components $X_0(M)$ and  $Y_0(M)$ are complex curves. If $M$ is not hyperbolic, our method may determine an algebraic set of representations of $M$.

\subsection{Lifting representations from $\PSL_2(\C)$ to $\SL_2(\C)$}\label{section:liftingreps}
Let $M=S^3-N(K)$ be a hyperbolic knot complement.  The details of the following discussion, that we briefly reproduce here, can be found in  \cite[\S 2.1]{MR1739217},\cite{MR1695208},   \cite[\S 3]{MR1670053}, \cite{MR1248117} and \cite{thurston}.

Let $\rho$ and $\rho'$ be two discrete faithful representations of $\pi_1(M)$ to $\PSL_2(\C)$. These representations must be conjugate in $O(3,1)$ by Mostow-Prasad rigidity  \cite{MR0236383, MR0385005},
but they may or may not be conjugate in $\PSL_2(\C)$.  Specifically,   $\rho'$ is conjugate in $\PSL_2(\C)$ to either $\rho$ or to $\overline{\rho}$.  The representation $\overline{\rho}$ is defined as follows.  For $\gamma\in \pi_1(M)$,  $\overline{\rho}(\gamma)$  is  entry-wise  the complex conjugate of $\rho(\gamma)$.  This difference corresponds to a choice of orientation of $M$ and the fact that $\PSL_2(\C) \cong \text{Isom}^{+}(\H^3)$. If two representations are conjugate in $\PSL_2(\C)$, they correspond to the same hyperbolic structure on $M$.

Culler proved that if a discrete subgroup $\Gamma$ of $\PSL_2(\C)$ has no 2-torsion, then it lifts to $\SL_2(\C)$ \cite{MR825087}.  That is, that there is a natural homomorphism from $\Gamma$ to $\SL_2(\C)$ which composed with the natural projection from $\SL_2(\C)$ to $\PSL_2(\C)$ is the identity on $\Gamma$.  Therefore, if we view an element $\gamma \in \Gamma$ as an equivalence class of matrices (since it is in $\PSL_2(\C)$), say  $(\pm) X$,  such a lift will take $(\pm) X$ to either $X$ or $-X$.  For representations of knot complements we can be more specific.

Let $\epsilon \in H^1(M,\Z/2\Z) \cong \Z/2 \Z$.  We can identify $\epsilon$ with a homomorphism from $\pi_1(M)$ to $\{I, -I\}$, where $I$ is the $2\times 2$ identity matrix.  Consider a representation $\rho:\pi_1(M)\rightarrow \PSL_2(\C)$ which lifts to $\tilde{\rho}:\pi_1(M)\rightarrow \SL_2(\C)$.  This gives another representation defined by  $\epsilon(\gamma) \tilde{\rho}(\gamma)$ for all $\gamma \in \pi_1(M)$.  In fact, the $\PSL_2(\C)$ character variety of $M$ is isomorphic to the $\SL_2(\C)$ character variety modulo $\Z/2\Z$ under this action defined by $\epsilon$.  One can see the homomorphism $\epsilon$ clearly for the Wirtinger presentation, $\Gamma\cong \pi_1(M)$ of the knot.  This  homomorphism defines a {\em parity} for any $\gamma\in   \Gamma$. That is, $\gamma$ can be written as a product of meridians in Wirtinger presentation, and the number of meridians modulo 2 is the parity.  This decomposes $\Gamma$ (and therefore $\pi_1(M)$ in general) into two cosets $\Gamma_e$ and $\Gamma_o$, the even and odd coset.  The kernel of $\epsilon$ is $\Gamma_e$.

Let $\tilde{\rho}_1$ and $\tilde{\rho}_2$ be two (different) lifts of $\rho:\pi_1(M) \rightarrow \PSL_2(\C)$ to $\SL_2(\C)$.  As stated above, for $\gamma \in \pi_1(M)$ we have $\tilde{\rho}_2(\gamma) = \epsilon(\gamma) \tilde{\rho}_1(\gamma)$.
For $\gamma\in \Gamma_e$, this is $\tilde{\rho}_2(\gamma) =   \tilde{\rho}_1(\gamma)$.
For $\gamma\in \Gamma_o$, this is $\tilde{\rho}_2(\gamma) =   -\tilde{\rho}_1(\gamma)$.
 The action of the homomorphism $\epsilon$ from $\Gamma$ to $\{I, -I\}$ induces an action on $X(\Gamma)$ by $\chi_{\rho}(\gamma) \mapsto \chi_{\epsilon(\gamma) \rho(\gamma)}$.
 The $\PSL_2(\C)$ representation variety for $M=S^3-N(K)$ is isomorphic to $X(\Gamma)/\epsilon$ under this action.

\subsection{Geometric Representations and Invariant surfaces in $\mathbb{H}^3$ }\label{section:geometricreps}

\begin{definition}
A {\em banana} is either a surface which consists of the points at a fixed hyperbolic distance from a geodesic in $\H^3$, or a horosphere. In the former case, we call the geodesic {\em the axis of the banana}, and the ideal points of the geodesic will be called {\em the ideal points (or endpoints) of the banana.} See Fig.\ref{Bananas} for a picture of such bananas in the upper half model of $\H^3$ that we use throughout. For a horosphere, we call its point of tangency with the boundary of $\mathbb{H}^3$ {\em the center}.
\end{definition}

\begin{figure}
\centering
\includegraphics[scale=0.7]{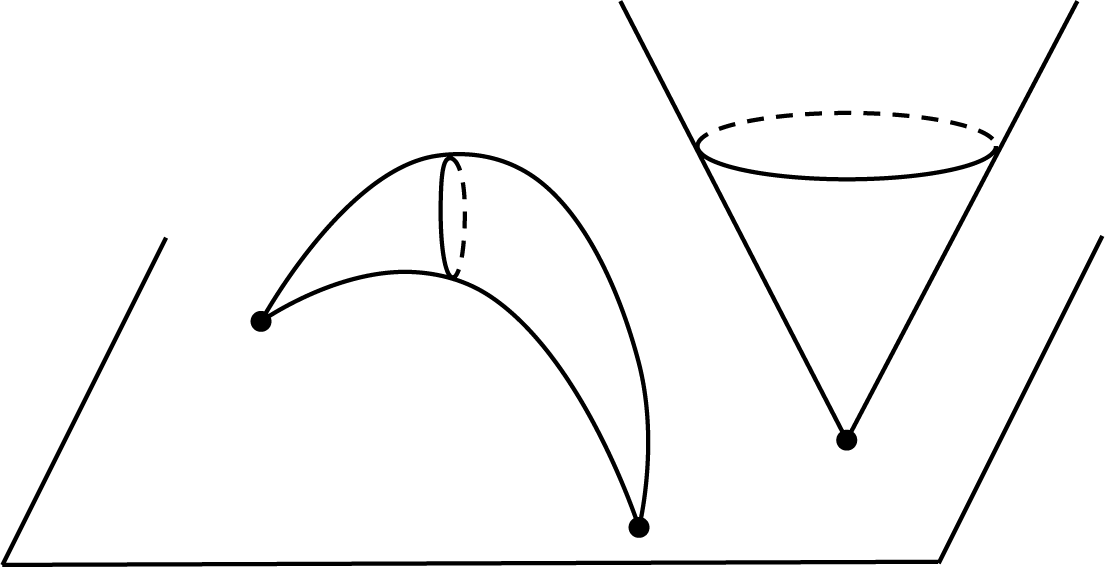}
\caption{Two bananas that are not horospheres.}
\label{Bananas}
\end{figure}

Thurston observed that if we consider the holonomy \[ \rho_0: \pi_1(M) \rightarrow \PSL_2(\C) \] of the complete hyperbolic structure, then  homomorphisms $\rho:\pi_1(M) \rightarrow  \PSL_2(\C)$  sufficiently close to $\rho_0$ are the holonomy of a (usually incomplete) hyperbolic structure on $M$ ( \cite[Chapter 4]{thurston}; also see p. 149 of \cite{Handbook}). These holonomies correspond to the geometric structures on $M$ where the boundary torus of a knot lifts to infinitely many (closed) bananas.  We will call such representations {\em geometric}. Note that $T$ will lift to an infinite collection of distinct horospheres if the representation is parabolic on $\pi_1(T)$. The holonomy representation of a complete hyperbolic structure is a discrete and faithful representation, and is parabolic. Two geometric representations which are conjugate in $\PSL_2(\C)$ correspond to the same hyperbolic structure on $M$.  Geometric representations which are conjugate after complex conjugation correspond to a change in orientation.

The set $H$ is defined in reference to a specific neighborhood of the knot, i.e. the specific cusp cross-section. For a  different choice of neighborhood, the resulting bananas are parallel copies of the original neighborhood,  enlarged or reduced in size. For horospheres in $H$, this change of the neighborhood is often referred to as ``blowing up", and the choice for which the horospheres touch but do not have overlapping interiors is called a ``maximal cusp". However, here we assume that the closure of the neighborhood of the knot does not intersect itself in more than one point. It follows that given a fixed neighborhood, two distinct bananas in $H$ must have distinct ideal endpoints. When the cusp torus intersects itself in only one point, our method still applies. This occurs, for example, for the complete structure of the figure-8 knot, when the meridian is scaled to be 1 \cite{MR1876890}, and this geometric set-up was used in \cite{MR3190595}.

From now on, we will fix orientation of the knot. We specify a complex affine structure on each banana by adopting a convention that the meridional translation is through unit distance in the positive
real direction. We will later see that in some cases, the orientation of meridian is also important. We therefore adopt the usual “right-hand screw” convention relating the directions
of meridian and longitude, when such a convention is needed (e.g. in direct computations later). For the complete parabolic structure, this will force the translation corresponding to a longitude on the
torus to have positive imaginary part.

\subsection{Zariski Closure.}\label{section:Zariski} If $M$ is a hyperbolic knot complement, then by work of Thurston, all but finitely many Dehn fillings of $M$ are hyperbolic \cite{thurston}. Moreover, in any neighborhood of a discrete and faithful representation of $\pi_1(M)$ to $\PSL_2(\C)$ there are infinitely many representations corresponding to these Dehn fillings.  These are all geometric representations. The fundamental group of any such Dehn filling is torsion-free and isomorphic to a discrete subgroup of $\PSL_2(\C)$, since it is hyperbolic.  Therefore, the representations into $\PSL_2(\C)$ corresponding to these hyperbolic structures on $M$ all lift in this natural way to $\SL_2(\C)$.   Since there is an infinite number of these representations in any neighborhood of a discrete and faithful representation of $\pi(M) \rightarrow \text{(P)SL}_2(\C)$, they form a Zariski dense subset of any canonical component (which is necessarily a curve).  Therefore to determine equations for the character variety of any canonical component, or to determine equations for the representations up to conjugation of a canonical component, it suffices to determine the equations governing  geometric representations.  {\it We will now only consider geometric representations.}

\section{Geometric Set-up: Peripheral and Crossing Arcs}\label{GeometricSetup}

In this section, we discuss the correspondence between certain arcs in the knot complement, $S^3-K$, and their preimages in $\mathbb{H}^3$ under the developing map. This can be seen as a generalization of Section 3 of \cite{MR3190595}. Later we will use this to establish the correspondence between the paths in the fundamental group of the knot complement, and the elements of $\PSL_2(\C)$.

\begin{definition}  Let $D$ be a diagram for a knot $K$.  An {\em overpass} in $D$ is a maximally connected portion of $K$, if one looks at $D$ from above.  An {\em underpass} is a maximally connected portion of $K$ if one looks at $D$ from below.  An {\em edge} is a connected portion of $D$ from one crossing to the next. If we thicken the knot, a {\em peripheral arc} is an arc lying on the boundary torus $T$ along one thickened edge.   A {\em crossing arc} is a cusp-to-cusp arc from an underpass to an overpass at a crossing.  Up to a homotopy, its preimage in $\mathbb{H}^3$ often lies on an ideal hyperbolic geodesic (see Theorem~\ref{Thm:crossingarc}). We refer to both the  cusp-to-cusp arc and the respective ideal geodesic in $M$ as a crossing arc denoting both by the same letter (a common abuse of notation). A {\em path} consists of peripheral and crossing arcs, connected into a simple (possible, closed) curve. Figure \ref{Path} depicts a thickened knot, and a (colored) path consisting of a red crossing arc and a green peripheral arc. A {\em region} $S$ of  $D$  is a disk in the plane whose boundary consists of edges and crossing arcs of $D$ as in Fig.\ref{Region} ($D$ is depicted in black, crossing arcs in grey), or alternatively, of peripheral and crossing arcs.
\end{definition}

\begin{figure}[h]
\centering
\begin{subfigure}[b]{0.53 \textwidth}
\centering
\includegraphics[scale=0.41]{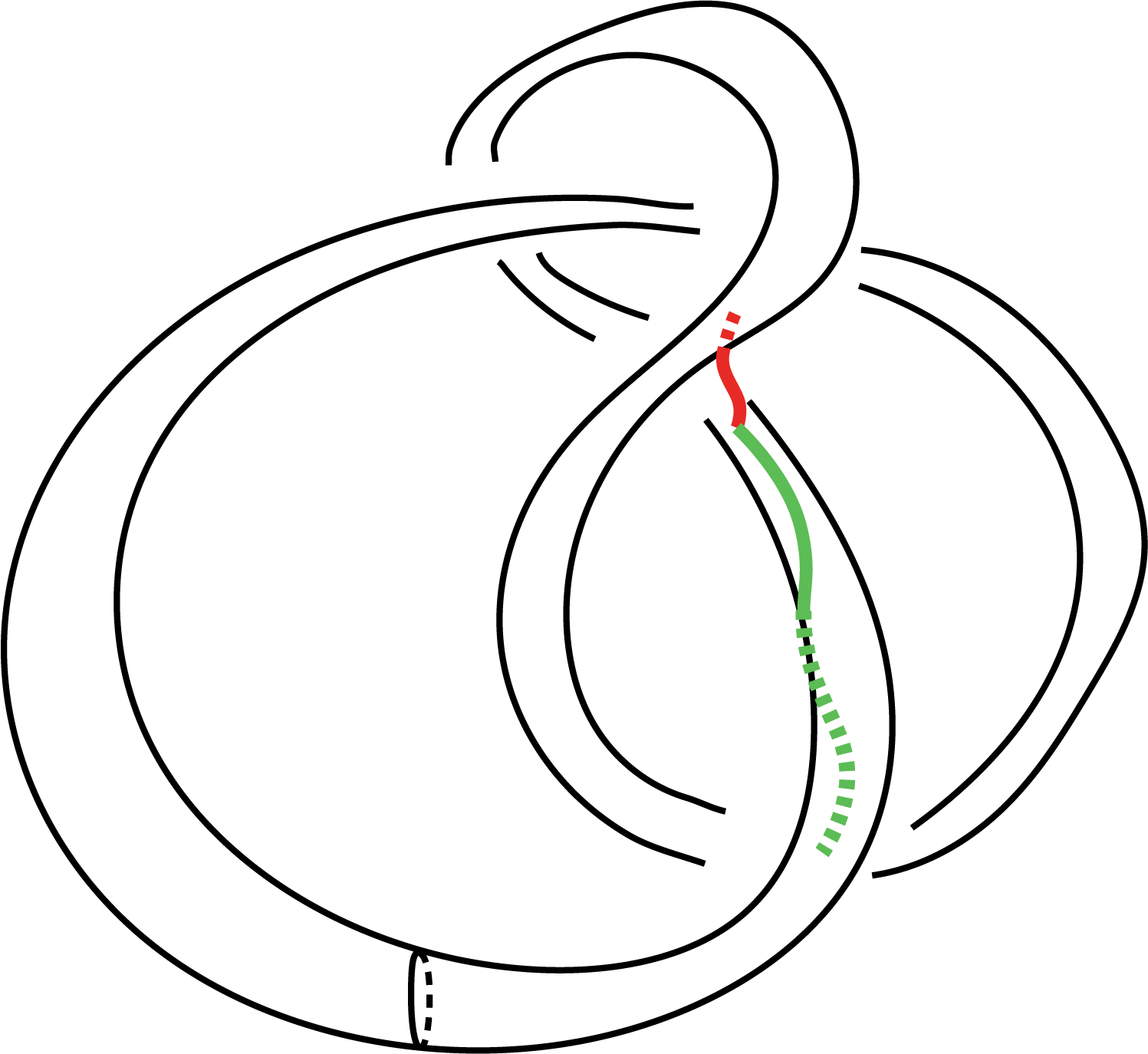}
\subcaption{}
\label{Path}
\end{subfigure}
\begin{subfigure}[b]{0.43 \textwidth}
\centering
\includegraphics[scale=0.68]{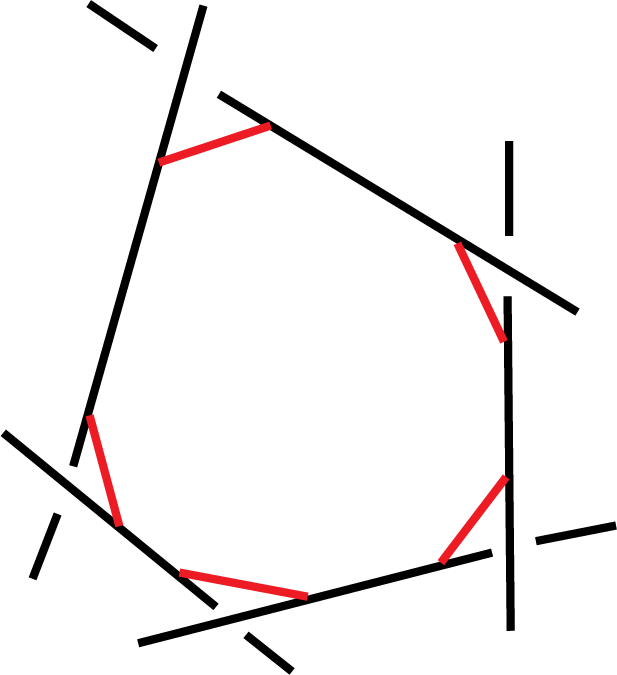}
\subcaption{}
\label{Region}
\end{subfigure}
\caption{Right: A path on a thickened knot, consisting of a crossing arc (in red) and a peripheral arc (in green). Left: A region of a knot diagram with five edges (in black). Five crossing arcs (in red) are also depicted.}
\end{figure}

\begin{remark}\label{isometries} For peripheral arcs  and meridians, we will use the following correspondence between the arcs and elements of $\text{PSL}_2(\C)$. Every horosphere locally resembles a Euclidean plane, and can be endowed with an affine structure. A translation on a horosphere corresponds to a parabolic isometry of $\mathbb{H}^3$.  For a non-parabolic representation, $T$ lifts to a collection of distinct bananas that are not horospheres. Such a banana is isometric to a Euclidean cone in $\H^3$ with ideal endpoints at $0$ and $\infty$, and therefore can be endowed with an affine structure as well. A translation along an arc in that structure may correspond to a hyperbolic, loxodromic or elliptic isometry of $\mathbb{H}^3$. For crossing arcs, the correspondence between them and elements of  $\text{PSL}_2(\C)$ is as follows. Ideal points of two bananas can be connected by a hyperbolic geodesic.  We will show that there is an infinite number of representations on a canonical component such that in $H$, some of these geodesics are preimages of crossing geodesics (and therefore preimages of crossing arcs, up to homotopy).  For two bananas and the hyperbolic geodesic connecting their ideal points, there is a unique  isometry that exchanges the bananas and keeps the geodesic fixed set-wise. Such an isometry is elliptic of order two. All these isometries of $\mathbb{H}^3$ leave the bananas invariant as hypersurfaces of $\mathbb{H}^3$. Further, we will show that these elements of $\text{PSL}_2(\C)$ are well-defined in Lemmas \ref{lemma:meridian}, \ref{lemma:extendtoperipheral}, and \ref{LemmaCrossinggeodesics}. 

\end{remark}

\subsection{Suitable knot diagrams.}\label{suitable}

Here and further denote a preimage of a path $\cdot$ in $M$  to a path in $H$ by $\tilde{\cdot}$. Our algorithm is based on lifting perhipheral and crossing arcs from the knot complement to $\H^3$, and assigning respective elements of $\text{(P)SL}_2(\C)$ to them. For this, we need to make sure that preimages of arcs in $\mathbb{H}^3$ are well-defined.

Recall that a checkerboard surface for a knot is a spanning surface which may or may not be orientable. In particular, a black (white) checkerboard surface is a union of black (respectively, white) disks obtained by coloring the regions of a knot diagram in a checkerboard fashion. The disks are connected through crossings by twisted bands.

We will consider topological accidental parabolics. For an embedded surface $S$ in $M$, a {\em topological accidental parabolic} is a free homotopy class of a closed curve that is not boundary parallel on $S$, but can be
homotoped to the boundary of $M$. This property is independent of the geometric structure of the 3-manifold. We will also consider incompressibility of embedded surfaces in topological sense, i.e. a surface is incompressible if it is neither a 2-sphere, nor contains any compressing disks.

The following definition was suggested by Thistlethwaite and Tsvietkova.

\begin{definition}
 A diagram of a hyperbolic link is {\em taut} if each associated checkerboard
surface is incompressible and boundary incompressible in the link complement, and
 does not contain any simple closed curve representing a topological accidental parabolic.
\end{definition}

We will show in Theorem~\ref{Thm:crossingarc} (2) that if $D$ is taut, then for infinitely many geometric representations, the lift of a crossing arc in $H$ is homotopic to a unique geodesic. While it is crucial for our algorithm, it is perhaps of independent interest as well, and there are related previous results. 

For hyperbolic alternating links, it was proven in Prop. 1.2 of \cite{MR3190595} that a reduced alternating diagram of a hyperbolic alternating link is taut as a consequence of \cite{MR721450} and \cite{Quasi}, and hence crossing arcs are not topological accidental parabolics. As a corollary, each crossing arcs is homotopic to a unique geodesic for the discrete and faithful representation, hence the method of Thislethwaite and Tsvietkova for computing the complete hyperbolic structure always works for a reduced alternating diagram. Here we establish this for other representations as well. 

In addition to hyperbolic alternating links and their reduced alternating diagrams, some other diagrams were found to be taut, and hence a priori suitable for the method. Two recent preprints show that for hyperbolic fully augmented links in 3-sphere and 3-torus, one can choose a set of geodesic crossing arcs in their fully augmented diagrams due to the existence of nice geometric decompositions of their complements \cite{Rocky, TorusTT}. State surfaces for hyperbolic adequate links were proved to be quasi-fuchsian in \cite{Quasi}, and hence one can choose cusp-to-cusp arcs with the necessary properties in an adequate link diagram.

Empirically, we have not yet seen a hyperbolic link that does not admit a taut diagram. We hence do not know of any knots to which our method would not be applicable. This can be compared with previously existing methods for computing varieties that use a triangulation of a 3-manifold. While it is not known how to algorithmically construct a suitable triangulation a priori, empirically, after some modifications, a suitable triangulation is always found by SnapPea kernel \cite{SnapPea}.

Note that when a surface is orientable, then the lack of accidental parabolics is known to imply that the surface is incompressible and boundary incompressible (a short topological proof was given in  \cite{Nathan}), and hence the definition of a taut diagram can be simplified. But checkerboard surfaces can be non-orientable. 
{\it From now on, we assume that we are working with a taut diagram.}

\subsection{Representation of a meridian.}\label{Meridian}  Recall that $\mu$ denotes a meridian of the knot. If the representation is parabolic, meaning that $\rho(\mu)$ is a parabolic element, then the corresponding isometry is a translation on a horosphere.  If $\rho(\mu)$ is hyperbolic/loxodromic or elliptic, then the corresponding isometry is a translation or a rotation along a banana with two distinct ideal points. An example of two such translations is given in Figure \ref{figure:wirtinger}. The lemma below shows that we can conjugate so that in all of these cases the respective matrix has a specific upper triangular form. Geometrically this means that there is a preimage of the meridian on a specific banana as remarked below. 

 \begin{figure}[h]
\centering
\includegraphics[scale=0.75]{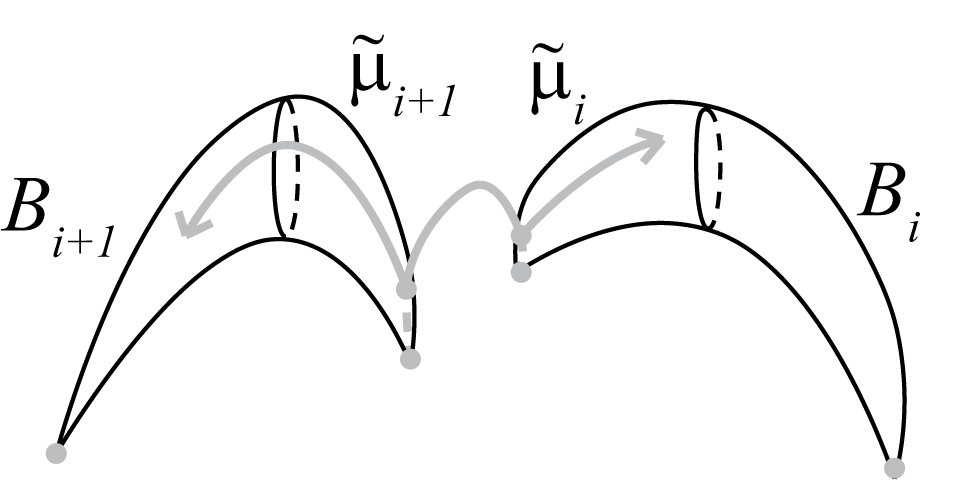}
\caption{The arcs $\mu_i, \mu_{i+1}$ on bananas $B_i, B_{i+1}$ are preimages of meridians.} 
\label{figure:wirtinger}
\end{figure}

\begin{definition}\label{definition:meridian}
Define a {\em meridian matrix}   $\Mx= (\pm) \mat{m}{1}{0}{m^{-1}} \in \PSL_2(\C)$, and let  $i_0=i_0(m)= -1/(m-m^{-1})$.  Let $I$ denote  $(\pm)$ the $2\times 2$ identity matrix.
  \end{definition}

 \begin{remark}\label{Mbanana}
 With this definition, the point $i_0$ is a fixed point of $\Mx$.  If $\Mx$ is a parabolic element, then $i_0=\infty$ is the only fixed point and is the ideal point of the horosphere at infinity.  Otherwise, if $\Mx$ is  not parabolic, then an invariant banana associated to $\Mx$ has fixed points $\infty$ and $i_0$. The point $i_0$ satisfies $(\text{tr}^2(\Mx)-4) i_0^2 = 1$.
 \end{remark}

\begin{lemma}\label{lemma:meridian}
Up to conjugation, we can take $\rho(\mu)=\Mx$.  Further, we can specify that  $|m|\geq 1$, and if  $|m|=1$ then $\arg(m)\leq \pi$.  This uniquely determines $\Mx$.
\end{lemma}

\begin{proof}

 First, we show that for a geometric representation, $\rho(\mu)\neq  I$.   Since  $\pi_1(M)$ is normally generated by $\mu$,  if   $\rho(\mu)= I$ then $\rho(\pi_1(M)) = \{ I\}$.  But as $\rho$ is geometric, $\partial M$ lifts to bananas in $\H^3$. This cannot be the case if $\rho(\pi_1(M))= I$, since one must be able to take one fundamental region to another using the isometries corresponding to the elements of $\rho$. An elementary calculation shows that any $X\neq  I$ in $\PSL_2(\C)$ can be conjugated to the form $\Mx$.

Since trace is a conjugation invariant, and $\rho(\mu)$ is conjugate to $\Mx$,  the $(1,1)$-entry of $\Mx$ is determined up to perhaps exchanging $m$ with $m^{-1}$.
If $|m|<1$, or $|m|=1$ and $\arg(m)>\pi$, an elementary calculation shows that one can further conjugate $\Mx$ so that $m$ and $m^{-1}$ are interchanged and the $(1,2)$-entry is still 1. Either $|m^{-1}|\geq 1$ or $|m^{-1}|=1$ and $\arg(m)\leq\pi$ respectively, so the lemma holds for the new matrix.
\end{proof}

This choice of conjugation makes $\rho(\mu)$ upper triangular with the specified $m$. This corresponds  to choosing a particular geometric arrangement for $H$.  Specifically, choose a base point $b$ on the cusp torus $T$, and let $\mu$ be a meridian based at $b$. Then conjugating so that $\rho(\mu)=\Mx$  corresponds to choosing a preferred lift of $b$ to lie on a banana   associated to $\Mx$, as in Remark \ref{Mbanana}.

\subsection{Representation of peripheral arcs}\label{PeripheralArcs}

Any representation $\rho$ is naturally defined for all elements of $\pi_1(M)$, that is for loops.  Often a representation can be extended, so that it is well-defined for certain paths as well.  Let $\rho$ be a  geometric representation, and $\alpha \in \pi_1(M)$ be a loop with base point $b$.  Using the developing map, $\alpha$ has a preimage $\tilde{\alpha}$  from $\tilde{b}$ to $\tilde{b}'$, two preimages of $b$ (where possibly $\tilde{b}=\tilde{b}'$).  The element $\rho(\alpha)$ then corresponds to the isometry of $H$ which sends $\tilde{b}$ to $\tilde{b}'$.  Now suppose $\alpha'$ is a path rather than a loop in $M$ from $b_1$ to $b_2$. We say that $\rho$ extends to $\alpha'$, if for a fixed lift $\tilde{\alpha'}$ of $\alpha'$ from $\tilde{b}_1$ to $\tilde{b}_2$, there is a unique isometry of $H$ taking $\tilde{b}_1$ to $\tilde{b}_2$. We denote by $\rho(\alpha')$ the associated element of $\text{PSL}_2(\C)$.

A \underline{simple meridian} $\mu_i$ is a loop freely homotopic to $\mu$, whose orientation agrees with $\mu$, based at a base point $b_i$.  We can view simple meridians as a special case of peripheral arcs.

\begin{lemma}\label{lemma:extendtoperipheral}
Let $\beta$ be a peripheral arc (possibly a simple meridian) in $M$.  If a $\PSL_2(\C)$ representation $\rho$ is not parabolic, $\rho$ uniquely extends to any preimage $\tilde{\beta}$ of $\beta$.  In the case when $\rho$ is parabolic, it uniquely extends given a specified meridianal direction.
\end{lemma}

\begin{proof}  With the developing map, a preferred  longitude $\lambda$ lifts to infinitely many paths in $H$, since $H$ covers $M$. Since the longitude is a loop based at $b$ in $M$, the choice of a preimage $\tilde{b}$ of $b$ determines a unique isometry of $H$.
The point $\tilde{b}$ lies on a banana in $H$, and this isometry fixes the banana as a set, fixing its ideal points. 

Let $\beta$ be an oriented peripheral arc in $M$, beginning at point $b_1$ and terminating at point $b_2$.  Then any chosen preimage $\tilde{b}_1$ of $b_1$ lies on a unique banana, and the corresponding preimage $\tilde{\beta}$ of $\beta$ terminates at a point $\tilde{b}_2$ which covers $b_2$.  The choice of the preimage $\tilde{b}_1$ determines $\tilde{\beta}$ uniquely. If the respective banana has distinct ideal points, then the specification that they are fixed, and $\tilde{b}_1$ is sent to a fixed lift of $b_2$ specifies a unique isometry. (This is the case when the isometry $\rho(\lambda)$ is hyperbolic/loxodromic, or elliptic.) If  $\rho(\lambda)$ is parabolic, then a unique isometry is determined if we specify a fixed direction along the horosphere, or if we specify the image of another point. This can be achieved by specifying the direction for the preimage of the meridian.
\end{proof}

\subsection{Representation of crossing arcs}\label{crossing_arcs_reps}

Now, we  determine what  a preimage of a crossing arc looks like in $H$  for a geometric representation and  specify the corresponding conjugacy classes of elements in $\PSL_2(\C)$. 
We often write $H(\rho)$ instead of $H$, to show that $H(\rho)$ is a cover corresponding to the representation $\rho$. 

 Suppose we have an arc $\tilde{\gamma}$ with ideal endpoints that are also ideal endpoints of bananas $B_1, B_2$ (with possibly $B_1=B_2$). We say that $\tilde{\gamma}$ {\em weaves through a banana} if for any homotopy $f(\tilde{\gamma}, t)$ between $\tilde{\gamma}$ and a geodesic, there is $t \in [0,1]$ such that $f(\tilde{\gamma}, t)$ intersects the geodesic axis of $B$, where $B\neq B_1$ and $B\neq B_2$. See Figure \ref{Weave}. Note that the property of weaving through a banana for an arc is independent from the size of the neighborhood of the knot.

\begin{figure}[h]
\centering
\includegraphics[scale=0.74]{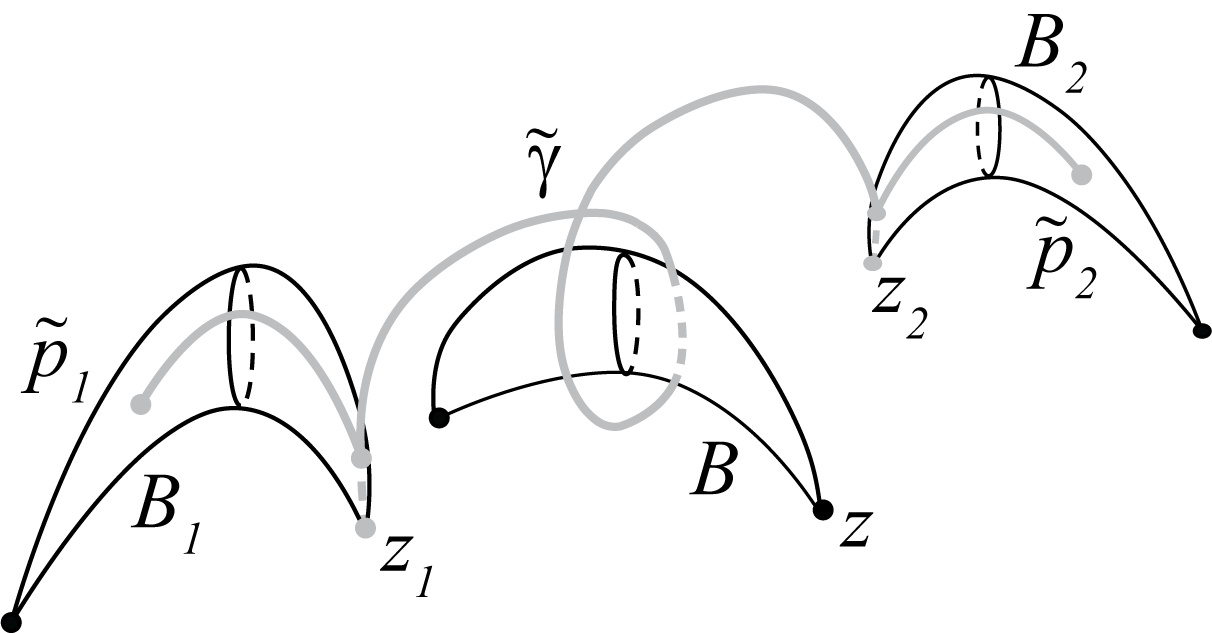}
\caption{The arc $\tilde{\gamma}$ from $\tilde{z}_1$ to $\tilde{z}_2$ weaves through banana $B$.}
\label{Weave}
\end{figure}

\begin{prop}\label{Thm:crossingarc}
Let $\gamma$ be a crossing arc in $M=S^3-K$, $\tilde{\gamma}$ a preimage of $\gamma$ in $H(\rho)$, and   $\rho_0$ be a discrete faithful representation of $M$.  Then the following holds.

\begin{enumerate}

\item\label{cross1} Any $\PSL_2(\C)$ representation  can be uniquely extended to $\tilde{\gamma}$.

\item\label{cross2} There are infinitely many geometric representations $\rho$ in any neighborhood of $\rho_0$ such that $\tilde{\gamma}$ is homotopic in $H(\rho)$ to a unique geodesic connecting two distinct bananas.

\item\label{cross3} Moreover, if the above homotopy is also an isotopy for $\rho_0$, i.e. in the complete hyperbolic structure, the same holds for those infinitely many geometric representations.

\end{enumerate}
\end{prop}

\begin{proof}

Let $p_1$ and $p_2$ be the initial and terminal points of  the cusp-to-cusp arc ${\gamma}$, i.e. $p_1$ and $p_2$ lie on the boundary torus. We can choose a peripheral arc $\beta$ from $p_2$ to $p_1$, making the concatenation $\beta \gamma$  a loop. Therefore, for any representation $\rho$, $\rho(\beta \gamma)$ is defined, and by Lemma~\ref{lemma:extendtoperipheral} so is $\rho(\beta)$ (if $\rho$ is parabolic, it is defined up to a meridianal direction, which we can fix).  As a result, $\rho(\gamma)= \rho(\beta)^{-1} \rho(\beta \gamma)$ is uniquely defined. This proves (\ref{cross1}).

Now we prove (\ref{cross2}). We consider $M=S^3-N(K)$ for a fixed neighborhood $N$ corresponding to a fixed horoball neighborhood of the cusp $C$  in the complete hyperbolic metric.  Let $\{ s_n\}$ be a sequence of slopes on $\partial C$ such that the length of their geodesic representative approaches infinity. Denote by $M(s)$ the manifold that results from a Dehn filling of $M$ along the slope $s$. By Thurston's hyperbolic Dehn surgery theorem (See \cite{purcell} Theorem 6.29 for this version with the geometric limit) for large enough $n$, the Dehn filled manifolds $M(s_n)$ are hyperbolic and approach $M$ as a geometric limit.

For $x\in M$, let $B_0(x, r)$  be the ball of radius $r$ about $x$ considered in the complete metric on $M$ and let $B_s(x, r)$ be the ball of radius $r$ about $x$ considered in the induced metric by $M\subset M(s)$.
As such, for all $\epsilon>0$ and all $r>0$ there exists an integer $N$ such that if $n>N$, then there  is a $(1+\tfrac1n, \epsilon)$-quasi-isometry
$f_n:B_{s_n}(x, r)  \rightarrow B_0(x, r)$. Therefore for any $y$ in $B_0(x,r)$ we have
\[
\frac{1}{(1+\tfrac1n)} d_{s_n}(x,y) -\epsilon \leq d_0(x,y) \leq (1+\tfrac1n) d_{s_n}(x,y)+\epsilon,
\]
where $d_0$ represents distance in $M$ with the complete metric, and $d_{s}$ represents distance in the metric induced from $M(s)$.  We conclude that for a path $\gamma$ in $M$, the length $\ell_{s}(\gamma)$ considered in the metric induced from $M(s)$ converges to the length $\ell_0(\gamma)$, considered in $M$ with the complete metric.

Let $\rho_s$ be the representation that corresponds to $s$.

 \textit{Claim 1.} For infinitely many $s$, no lift  $\tilde{\gamma}$ in $H_{\rho_s}$  of the cusp-to-cusp arcs $\gamma$  in $M$  weaves through  bananas.

Proof of claim: assume  the claim does not hold. Suppose in $H_{\rho_s}$, the lift $\tilde{\gamma}$  has initial and terminal points on bananas $B_{1,s}$ and $B_{2,s}$, where these are possibly the same banana. To arrive to a contradiction, it is enough to show that in the geometric limit, the length $\ell_{s}(\gamma)$ gets arbitrarily large, showing that for such $s$ there is no weaving.  Let $B_{3,s}$ denote a banana that $\tilde{\gamma}$ weaves through. Up to a conjugation, we may assume that $B_{1,s}$ has an ideal point at $(0, 0)$, and $B_{3,s}$ has an ideal point at $(0, 1)$ on the plane $z=0$ in the upper half-space model of $\mathbb{H}^3$.  Then in the limit, the other ideal point of $B_{3,s}$ approaches 1 as well.  Hence $\ell_{s}(\gamma)\rightarrow \infty$, since $\tilde{\gamma}$ gets arbitrarily close to $\partial \H^3$. This concludes the proof of Claim 1.

We now show that $\tilde{\gamma}$ cannot have initial and terminal points on the same banana for infinitely many $s$.  If  it was true,   apply Claim 1.  Then for infinitely many $s$ from the claim, $\tilde{\gamma}$ is homotopic rel endpoints to a curve lying on a banana.  It follows that $\gamma$ is homotopic rel endpoints to a curve on $\partial M$ which contradicts the tautness assumption.

Therefore, we can assume that $\tilde{\gamma}$ has initial and terminal points on different bananas. Then $\tilde{\gamma}$ is homotopic to a geodesic in $\H^3$.  To show that $\tilde{\gamma}$  can  be homotoped in $H(\rho_s)$ to a unique geodesic it then suffices  to show that  the axis of $\tilde{\gamma}$ does not intersect the axis of any banana. Let $G\in \H^3$ denote the unique geodesic line containing $\tilde{\gamma}$ with ideal endpoints $z_1$ and $z_2$. Using Claim 1, we may assume that $\tilde{\gamma}$ does not weave through any bananas.

\textit{Claim 2.} For infinitely many $s$ satisfying Claim 1 the following holds. For any lift $\tilde{\gamma}$ in $H_{\rho_s}$ of a cusp-to-cusp arc $\gamma$, $G$ does not intersect the axis of a banana.

  Proof of Claim 2.
    Up to conjugation, we can take $z_1=\infty$, and assume that $z_1$ and $z_2$ remain fixed in the geometric limit, when $M(s_n)$ are hyperbolic and approach $M$ as above. In this limit,  any $B$ converges to a horosphere. This implies that in the limit, the two ideal endpoints of banana $B_{i,s}, i=1, 2,$ get closer together in the Euclidean distance on $\partial\H^3$. Since $z_1$ and $z_2$ are fixed, both endpoints of any banana $B$ whose axis intersects $G$ must similtaneously approach one of the endpoints $z_1$, or $z_2$. But this implies that in the limit, both ideal endpoints of  any such $B$ must coincide with the endpoints of $B_{i,s}$ for $i=1$ or $2$. This cannot occur, since we showed that $\tilde{\gamma}$ cannot have initial and terminal points on the same banana. This concludes the proof of Claim 2.

This shows that in the geometric limit, there are infinitely many Dehn filling representations $\rho$ so that for all crossing arcs $\gamma$, any lift $\tilde{\gamma}$ in $H(\rho)$ is homotopic to a unique geodesic connecting two distinct bananas, as desired.

If this is an isotopy for $\rho_0$ then there are no self-intersections in the homotopy of the arc.  As this is independent of the geometric structure, the same holds for the infinitely many representations as above. This proves (\ref{cross3}).\end{proof}

The following lemma shows that the matrix determined by a crossing arc $\gamma$ is independent of the choice of neighborhood of the knot, and therefore independent of the choice of $H$.

We call two bananas \textit{nested}, if they share all ideal points.
 If such bananas are distinct, they correspond to  differently scaled neighborhoods of a knot.

\begin{lemma}\label{LemmaCrossinggeodesics}
Let  $\gamma$ be a crossing arc  in $M$, and $\tilde{\gamma}$ be a lift of $\gamma$ to $H$.  Up to conjugation, the element in $\PSL_2(\C)$ determined by $\tilde{\gamma}$ (as in the Remark \ref{isometries}) has  the form $ (\pm)  \mat{0}{c}{-c^{-1}}{0}$. Furthermore, $c$ is uniquely determined up to a sign if we specify that the isometry preserves the meridianal direction.
\end{lemma}

\begin{proof} We give an orientation to the taut diagram.   Fix a cusp neighborhood of the knot.
Let $\gamma$ be a crossing arc on $T$ between points $p_1$ and $p_2$ on the cusp.
  By Theorem~\ref{Thm:crossingarc},  up to homotopy we may assume   that $\tilde{\gamma}$ is a geodesic.  Let $z_1$ and $z_2$ be   its two ideal endpoints, and $\tilde{p}_1$ and $\tilde{p}_2$ its two non-ideal endpoints  on bananas. Let $X$ be a matrix defined by the isometry $x$ of $H$ that sends $\tilde{p}_1$  to $\tilde{p}_2$.

  By Theorem~\ref{Thm:crossingarc}, $z_1, z_2$ belong to distinct bananas, say $B_1, B_2$ respectively.  If $B_i$ is not a horoball, let $z_i'$ be the other its ideal point. We will show that $x$ is independent of the choice of the neighborhood of the knot. Taking smaller or larger neighborhoods of the knot results in smaller or larger  bananas, nested in $B_1$ and $B_2$, with the same ideal points. The points $\tilde{p}_1$ and $\tilde{p}_2$  approach the ideal points, $z_1$ and $z_2$ respectively, upon making the cusp neighborhood smaller. Therefore,  the isometry  $x$ must take $z_1$ to $z_2$, and therefore take $B_1$ to $B_2$. As a result, if the bananas are not horoballs, the isometry must take $z_1'$ to $z_2'$ as well.  (These are distinct bananas, and since $H$ is a cover, they must be disjoint.)    Since the isometry $x$ takes $\tilde{p}_1$ to $\tilde{p}_2$, and $z_1$, $\tilde{p}_1$, $z_2$, and $\tilde{p}_2$ are on $A$, we conclude that it fixes the ideal geodesic $A$ from $z_1$ to $z_2$ set wise. Since $x$ takes $z_1$ to $z_2$, and fixes both $A$ and $\partial \mathbb{H}^3$ setwise, it must take $z_2$ to $z_1$.

  When the bananas are not horoballs, i.e. in the non-parabolic case, the choice of a different neighborhood clearly results in the same isometry $x$ as for larger bananas and the same matrix $X$. Upon conjugating so that $z_1=\infty$ and $z_2=0$, we see that the isometry $x$ corresponds to a matrix of the form
 \[ X =(\pm) \mat{0}{c}{-c^{-1}}{0}.\]
If $f$ is a fixed point of $x$, a simple computation  shows that $f=\pm c i$. Hence a fixed point $f$ determines the modulus of $c$.  (Fixed points can be computed directly from the fact that the isometry takes $\tilde{p}_1$ to $\tilde{p}_2$ as well.)  The argument of $c$, up to $\pm n\pi$, can be determined  by the fact that $B_1$ is taken to $B_2$. Alternatively, it can be determined from the fixed point $f$. More generally,  $X$ is uniquely determined up to sign by the  dihedral angle between two meridians that start at points of intersection of $\tilde{\gamma}$ and $B_1, B_2$.

 In the parabolic case, choosing a different neighborhood results in nested horoballs, so one can still conjugate to a matrix of the above form.  It is enough to see that $c$ is uniquely determined up to sign.
 With $B_1$ a horoball at infinity, and $B_2$ a horoball with ideal point at $0$, we consider the horosphere boundary of $B_1$ and the tangent plane to $B_2$ at $\tilde{p}_2$, both of which are horizontal planes.  On these planes, consider a vector beginning at $\tilde{p}_i$ in the meridian direction. In particular, on $\partial B_1$, the meridian lifts to paths in a Euclidean plane, and we consider the vectors from the initial to terminal points of these paths in the Euclidean structure on the horosphere. This is well-defined after choosing an orientation for the meridian. The isometry will take $\tilde{p}_1$ to $\tilde{p}_2$ and exchange these vectors.  Nested horoballs will yield vectors in the same direction. Therefore, the dihedral angle between these vectors in the two planes determines the argument of $c$ up to $\pm \pi n$.
\end{proof}

\begin{remark}\label{altW} The above shows that the crossing arcs (up to conjugation) are given by  $(\pm) \mat{0}{c}{-c^{-1}}{0}$. Such a matrix squares to $\pm I$.  
\end{remark}

Consider two crossing arcs $\gamma_1, \gamma_2$ with a single peripheral arc $\beta$ between them, and their preimages $\tilde{\gamma}_1, \tilde{\gamma}_2, \tilde{\beta}$ in $\mathbb{H}^3$. For a parabolic representation, $\tilde{\beta}$ is a Euclidean geodesic on a horosphere in $H$, and the center of this horosphere is an ideal point that belongs to both geodesics on which $\tilde{\gamma}_1$ and $\tilde{\gamma}_2$ lie. For a non-parabolic representation, the arc $\tilde{\beta}$ lies on a banana $B$ that is not a horosphere, with two ideal points $z_1$ and $z_2$.  Fig.\ref{CrossingArcs2} shows the situation that Lemma~\ref{LemmaIdealpointsofbananas} proves {\it does not} occur, and Fig.\ref{CrossingArcs3} shows the situation that {\it does} occur.  Note that the geodesic corresponding $\tilde{\gamma}_1$ or $\tilde{\gamma}_2$ cannot coincide with the axis of any banana on the figure, as otherwise the crossing arc would be contained in every neighborhood of the knot. Therefore, shrinking the neighborhood of the knot does not change geodesic connecting the endpoints of two bananas.

\begin{figure}[h]
\centering
\begin{subfigure}[b]{0.47 \textwidth}
\centering
\includegraphics[scale=.7]{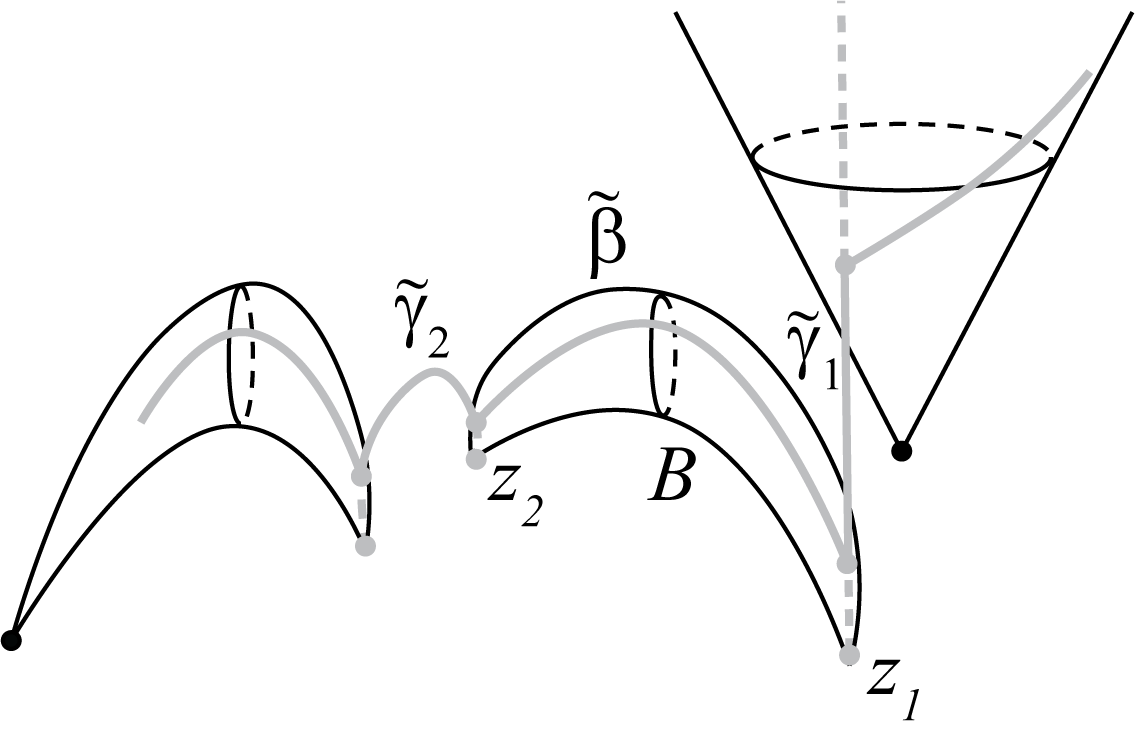}
\caption{By Lemma~\ref{LemmaIdealpointsofbananas}, this {\it does not} occur.}\label{CrossingArcs2}
\end{subfigure}
\begin{subfigure}[b]{0.47 \textwidth}
\centering
\includegraphics[scale=.69]{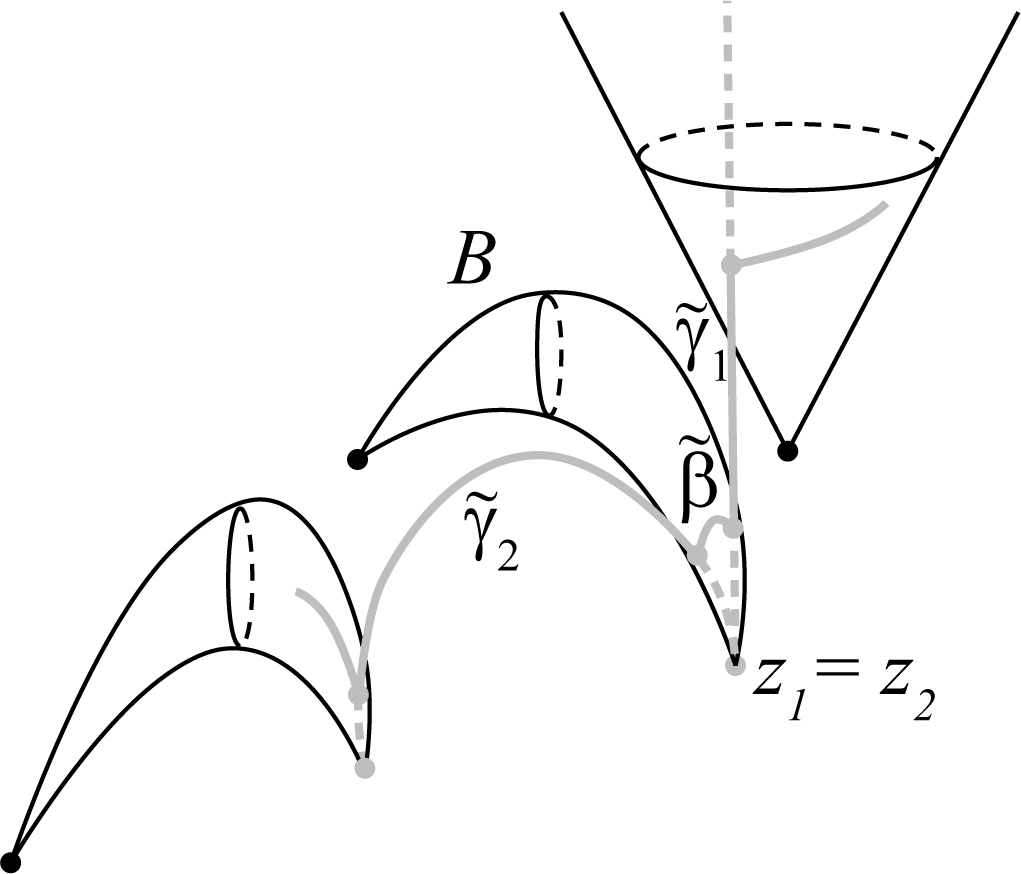}
\caption{By Lemma~\ref{LemmaIdealpointsofbananas}, this {\it does} occur. }\label{CrossingArcs3}
\end{subfigure}
\caption{A lift of $\beta$ to $H$, with $z_1\neq z_2$ on the left, and  $z_1=z_2$ on the right.}
\label{z1=z2}
\end{figure}

\begin{lemma}\label{LemmaIdealpointsofbananas}
Consider two crossing arcs which occur with only a single peripheral arc between them. The preimages of these crossing arcs share an ideal point.
\end{lemma}

\begin{proof}
It suffices to consider the case when the bananas are not horospheres.
In the language established above,  we need to show that $z_1=z_2$ (see Fig.\ref{z1=z2}). Suppose that $z_1\neq z_2$. Regardless of the neighborhood of the knot that is chosen, $\tilde{\beta}$ intersects   $\tilde{\gamma}_1$ and $\tilde{\gamma}_2$, since $\beta$ intersects $\gamma_1$ and $\gamma_2$. If we take smaller neighborhoods of the knot, then these intersection points (the starting and ending points of $\tilde{\beta}$) get closer to $z_1$ and $z_2$, respectively, in the Euclidean metric in the upper-half space. Since we are assuming that $z_1\neq z_2$, this means for any distance $d$, there is some  neighborhood of the knot such that  these two intersection points are greater than distance $d$ apart in that neighborhood, where $d$ is the hyperbolic distance. Simultaneously, the Euclidean length of $\tilde{\beta}$ grows as well.

But the peripheral arc $\beta$ has a well-defined translation distance.

Indeed, regardless of the neighborhood, the geodesic path $\tilde{\beta}$ corresponds to a fixed non-parabolic isometry in $\PSL_2(\C)$ by Lemma \ref{lemma:extendtoperipheral}, with a fixed trace. Hence $\tilde{\beta}$ corresponds to a translation of a fixed distance, and a fixed rotation angle. Up to conjugation, the matrix corresponding to the isometry has the form $X=(\pm) \mat{\ell}{0}{0}{\ell^{-1}}$. Then the complex length of $\beta$ in $M$  is  $\ell_0+i\theta$ where $\ell_0=2\log |\lambda|$ with
\[ \lambda =\frac12 ( \tr X \pm \sqrt{ (\tr X)^2 -4}) =\frac12  ( (\ell+\ell^{-1} ) \pm (\ell-\ell^{-1}))= \ell \text{ or } \ell^{-1}  \] 
 and   where $\theta$ is the angle of rotation (see, for example, Lemma 12.1.2 in \cite{MR1937957}). Alternately, $\text{cosh}(\ell/2)=\pm \tr X/2$.  \
We conclude that $z_1=z_2$.
\end{proof}

\section{Geometric Set-Up: Paths}\label{Paths}

To define a representation $\rho$ from $\pi_1(S^3-K)$ to  $\mathrm{PSL}_2(\C)$, we need only define $\rho$ for Wirtinger generators, and show that the Wirtinger relations hold.

 Call a path in $M=S^3-K$ {\em a path on the top of the knot} if it is the concatenation of peripheral arcs and crossing arcs, and can be homotoped   to lie entirely above the projection plane for the knot.  That is, if the path can be pulled up off the surface of the knot and is not interwoven with the knot.   Fig.\ref{CrossingArcs6} and \ref{CrossingArcs4} show fragments of such paths in grey, while the cusp neighborhood (thickened knot) is in black.    We will refer to paths which lie entirely on $T$ (and their lifts) as {\em peripheral paths}.

\begin{figure}[b]
\begin{subfigure}[b]{0.49 \textwidth}
\centering
\includegraphics[scale=1.4]{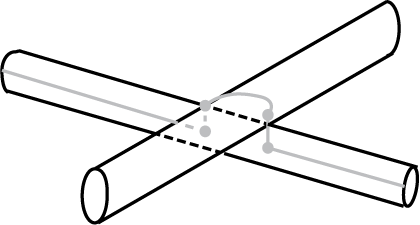}
\caption{A fragment of a path along the top of the knot.}
\label{CrossingArcs6}
\end{subfigure}
\begin{subfigure}[b]{0.49 \textwidth}
\centering
\includegraphics[scale=0.9]{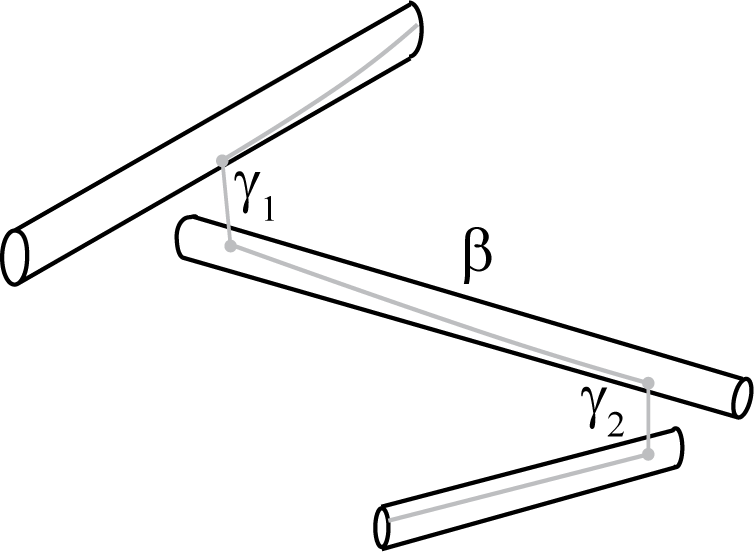}
\caption{The path $\alpha$.}
\label{CrossingArcs4}
\end{subfigure}
\caption{}
\end{figure}

For a region $S$ of a diagram $D$, consider a loop $\alpha_S$ on top of the knot that follows the boundary of $S$ and is homotopically trivial.  Assume $\alpha_S$ is the concatenation of arcs alternating between peripheral and crossing ones. We call $\alpha_S$ a {\em loop associated to the region $S$}.

Recall that $b$ is a base point on $\partial M$, and $\mu$ is a meridian around the overpass that $b$ lies on. Here and further we will introduce more basepoints, denoting them by $b_i$.

Two overpasses, or the two corresponding Wirtinger generators, or an overpass and an underpass, or two peripheral arcs are \underline{adjacent} if they meet at a crossing. Two such overpasses are pictured in Fig.\ref{FigureAdjacentarcs} in black color, and two adjacent peripheral arcs in grey color.

\begin{figure}[h!]
\centering
\includegraphics[scale=0.9]{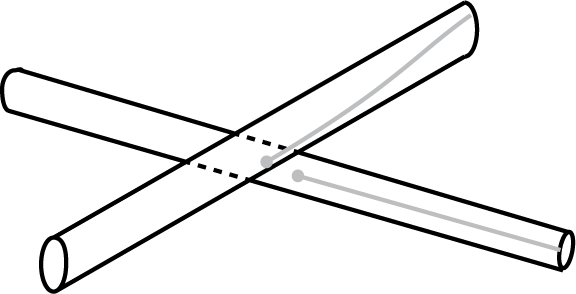}
\caption{Two adjacent peripheral arcs.}
\label{FigureAdjacentarcs}
\end{figure}

\subsection{Normalizing Wirtinger Generators}\label{section:setupPSL} 
In this subsection, we show that up to a conjugation one can simultaneously specify matrices of a particularly nice form for $\rho(\mu)$ and $\rho(w_q)$, if $q$ is adjacent to the overpass where $b$ and $\mu$ lie.

By the triple transitivity of $\PSL_2(\C)$, there are three degrees of freedom to fix a conjugacy class representative of a representation.  In Lemma~\ref{lemma:meridian} we use two of these to fix the format of $\rho(\mu)$ for our preferred meridian $\mu$.  The proposition below uses the remaining degree of freedom to fix a preferred meridian adjacent to $\mu$ as lower triangular. With these choices, the conjugacy class representative is now determined.

 \begin{prop}\label{PropositionConjugatingmeridianandwirt}
Let $\rho:\pi_1(M) \rightarrow \mathrm{PSL}_2(\C)$ be a geometric representation, and  let the overpass $q$ be adjacent to the overpass with the base point $b$. Then up to simultaneous  conjugation \[ \rho(\mu)= \Mx \  \text{and} \ \rho(w_q)=(\pm)\mat{m^{-1}}{0}{d}{m}.\]
We can  choose $|m|\geq 1$ and if $|m|=1$ then $\arg(m)\leq \pi$.  Unless $m=\pm i$, $\Mx$ and $\rho(w_q)$ are unique.

Moreover, this normalization corresponds to taking a lift of $\mu$ to lie on a banana with ideal  point(s) $i_0$ and $\infty$ ($i_0=\infty$ when $m=\pm 1$), and taking a lift of $\mu_q$ to be on a banana with an ideal point at 0. The crossing arc between the respective overpasses then lifts to a geodesic from 0 to $\infty$.

\end{prop}

 \begin{proof}

By Lemma~\ref{lemma:meridian} and Remark \ref{Mbanana}, the element $\rho(\mu)$ has the above matrix form, and this corresponds to taking a preferred banana $B_1$ to have ideal points $i_0$ and $\infty$. Let a crossing arc $\gamma$ start at $b$. We can realize $w_q$ as $\gamma^{-1} \mu_q \gamma$. If $b'$ is the other non-ideal endpoint of $\gamma$, then $\mu_q$ is a meridian around the overpass that $b'$ lies on. We can assume that the crossing arc $\gamma$ lifts to a hyperbolic geodesic $\tilde{\gamma}$ which goes from $B_1$ to a second banana $B_2$ by Theorem~\ref{Thm:crossingarc}.  We can normalize so that $B_1$ is as above and $B_2$ has an ideal vertex at 0.  i.e. so that the geodesic  $\tilde{\gamma}$ has ideal endpoints at 0 and $\infty$.
By Lemma~\ref{LemmaCrossinggeodesics}, the element of $\PSL_2(\C)$ associated to this lift of $\gamma$ has the form $(\pm) \mat{0}{c}{-c^{-1}}{0}$, and is uniquely determined.   Therefore, the element in $\PSL_2(\C)$ associated to $w_q$ is
\[ (\pm) \mat{0}{c}{-c^{-1}}{0} \Mx \mat{0}{c}{-c^{-1}}{0} = (\pm) \mat{m^{-1}}{0}{-c^{-2}}{m}. \]

Uniqueness follows from the uniqueness of $\Mx$ and $(\pm)\mat{0}{c}{-c^{-1}}{0}$.
 \end{proof}

 \begin{remark}\label{RemarkNonuniqueness}
As stated, the format used in Proposition~\ref{PropositionConjugatingmeridianandwirt} is unique unless $m=\pm i$. Here we remark on this case.
 Two (equivalence classes of) matrices of the same trace which span an irreducible subspace can be conjugated into the form
\[ (\pm) \mat{m}{1}{0}{m^{-1}} \text{ and  } (\pm) \mat{m^{-1}}{0}{d}{m}\]
and this is unique unless $m=\pm i$. If $m=\pm i$,  we can conjugate by matrices of the form
 \[ (\pm ) \mat{\sqrt{(d+4)/d}}{-2i/\sqrt{d(d+4)}}{0}{\sqrt{d/(d+4)}} \]
 which will take $\Mx$ and $\rho(w_2)$ type matrices to matrices of the form
 \[
 (\pm) \mat{i}{1}{0}{-i}= \text{ and } (\pm) \mat{i}{0}{4+d}{-i}  \]
respectively.

There are only finitely many conjugacy classes of representations which are traceless for the meridian  (that is, $\tr(\rho(\mu))=0$) on a  canonical component for a hyperbolic knot complement. Otherwise, the traceless condition would hold on a component of dimension at least one on the character variety, as it is an algebraic condition in the traces. Since the canonical component of the character variety of a knot is a curve (see Section \ref{Background}), an infinite number of traceless characters would necessitate that all characters are traceless.  However, a discrete and faithful representation takes the meridian to a parabolic matrix with trace $\pm 2$.

\end{remark}

\subsection{Degenerate Geometries.}

 We now prove a    non-degeneracy lemma.

\begin{lemma}\label{LemmaDegeneracy1}
Let $S$ be a region in a taut diagram $D$ of $K$, with more than two sides, and $\alpha_S$ be a loop associated to $S$.  Then $\alpha_S = ...\gamma_1\beta\gamma_2...$, where $\beta$ is a peripheral arc, and $\gamma_i, i=1,2,$ are crossing arcs. Suppose in a fundamental region for $S^3-K$, preimages $\tilde{\gamma_1}, \tilde{\gamma_2}$ intersect bananas $B_0, B_1$ and $B_1, B_2$ respectively. If $\rho(\beta)= I$, then $\rho(\gamma_1)= \rho(\gamma_2)$, and $B_0=B_2$.
\end{lemma}

Fig.\ref{LemmaTrivPeriph} shows the situation discussed in the lemma in the {\it non-degenerate case}, when the bananas $B_0$ and $B_2$ are distinct. The lemma gives a sufficient algebraic condition for the geometry to be degenerate, meaning $B_0=B_2$.

\begin{figure}[h!]
\centering
\includegraphics[scale=.75]{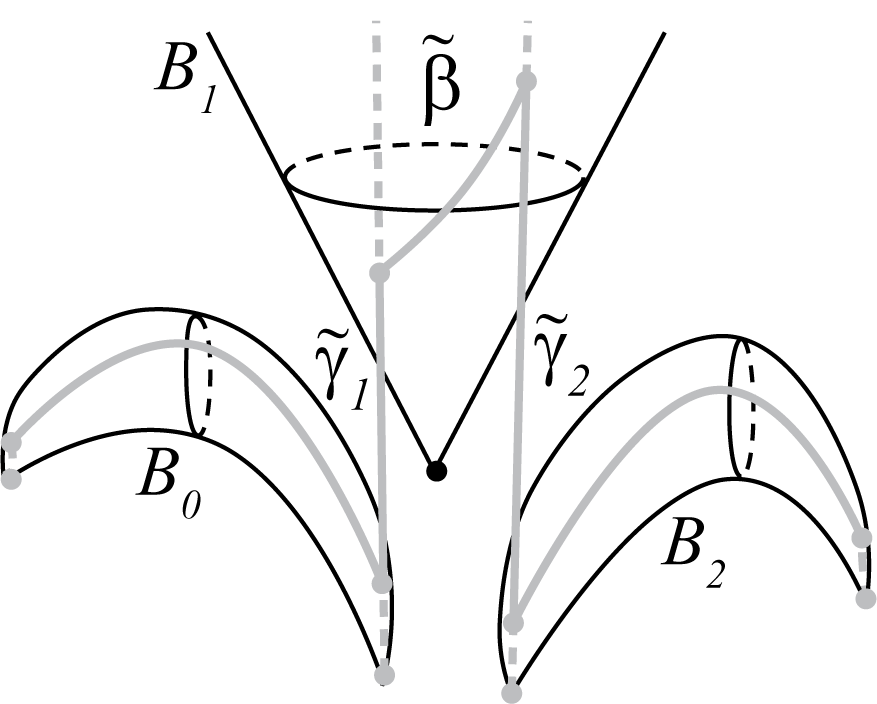}
\caption{A lift of $\alpha$ to $H$ in a non-degenerate case.}
\label{LemmaTrivPeriph}
\end{figure}

\begin{proof}

Perform an isometry of $\mathbb{H}^3$ placing $B_1$ so that its ideal points are $\infty$ and $i_0$, i.e. $B_1=B_{\infty}$.

Since $\rho(\beta)=I$, the arc $\beta$ is trivial, and  the two intersection points of $\tilde{\gamma_i}$ with $B_1$ are the same. This is true regardless of how small the neighborhood of $K$ is.  We conclude that  $\tilde{\gamma_1}, \tilde{\gamma_2}$ must be the same ideal geodesic.

 The banana $B_0$ is pierced by $\tilde{\gamma_1}$ in one point, and therefore must have ideal point (say $0$) in common with the ideal geodesic $\tilde{\gamma_1}$.  Also, $B_2$ must have an ideal point in common with $\tilde{\gamma_2}$.  It follows that $B_0, B_2$ share an ideal point.  But then $B_0=B_2$.  Otherwise their interiors intersect, or they are parallel copies of each other.  The interiors cannot intersect by our choice of cusp neighborhood. They cannot be parallel copies since this would imply that the same crossing arc pierces $T$ twice  on one side of the arc.  That is, the geodesic arc pierces $T$ a total of at least 3 times.

 It follows that the arcs $\gamma_1$ and $\gamma_2$, up to orientation are the same.  Since $\rho(\gamma_i)$ is an involution by Lemma~\ref{LemmaCrossinggeodesics}, the result follows.

\end{proof}

\subsection{Representing Paths.}\label{definingrepresentations}

We now show that our  representation  can be extended to  paths along the top of the knot using  representations of peripheral and crossing arcs, as well as meridians. In what follows, it is useful to have labeling of all arcs. 

\begin{definition}\label{remark:naturallabeling} 

\begin{enumerate}
\item  Let $D$ be an oriented taut diagram of the knot $K$, with $n$ crossings. Denote the edges of $D$ by $E_1, \dots ,E_{2n}$ according to the orientation of $K$, starting from a fixed basepoint $b$. We may assume that  $b$ is an endpoint of a crossing arc.  Define two base points, $b_i$ and $b_i'$, to be the intersections of $E_i$ with the first and the second crossing arcs, according to the orientation of $E_i$.  Let $\beta_i$ be the peripheral arc beginning at $b_i$ and traveling in the direction agreeing with the orientation of $K$ to $b_i'=b_{i+1}$.  This induces a labeling on the crossing arcs so that $\gamma_i$ is the crossing arc beginning at $b_i$. We will call this the \underline{natural labeling} of arcs. Note that since there is no orientation for crossing arcs, each crossing arc $\gamma_i$ in a taut knot diagram will have two different labels, $\gamma_i=\gamma_j$ for some $j$, where the crossing arc is adjacent to the edges $E_i, E_j$.  

 \item We will distinguish between {\em left} and {\em right} peripheral arcs $\beta_{i,L}$ and $\beta_{i,R}$ from $b_i$ to $b_i'$,  as in Figure \ref{LeftRight}.  They travel on  left or right side of a thickened edge respectively.

\begin{figure}[h!]
\centering
\includegraphics[scale=0.88]{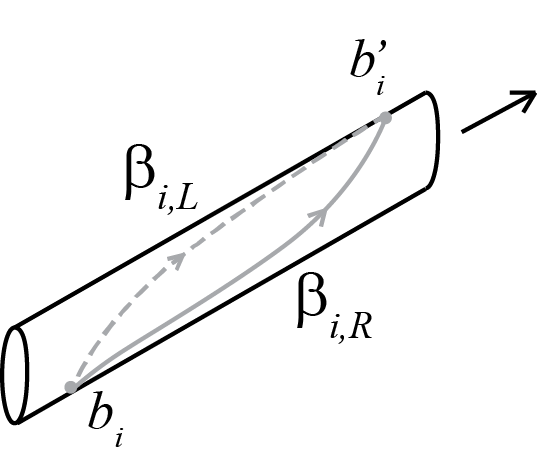}
\caption{Left and right paths}
\label{LeftRight}
\end{figure}

\item For an oriented peripheral arc $\beta$, we write $\beta^{-1}$ to indicate $\beta$ with the orientation reversed. By Lemma~\ref{LemmaCrossinggeodesics}, for any geometric representation $\rho$, the representation of a crossing arc  $ \rho(\gamma_{ij})$ is an involution, and hence $\rho(\gamma_{ij})=\rho(\gamma_{ij})^{-1}$.  We therefore will not orient the crossing arcs or use $-1$ exponents for them.
\end{enumerate}
\end{definition}

As first observed by Thislethwaite and detailed in Section 3 of \cite{MR3190595},
\begin{equation}\label{EquationEdges1} \beta_{i,L} = \mu^{g_i} \beta_{i,R}\end{equation}
where \begin{equation}\label{EquationEdgesExp}
g_i = \begin{cases}
 \phantom{-}1, &\text{if $E_i$ travels from an overpass to an underpass;}\\
	   {-} 1,      &\text{if $E_i$ travels from an underpass to an overpass;} \\
   \phantom{-}0,       &\text{if $E_i$ travels from an overpass to an overpass} \\
    & \quad \text{or from an underpass to an underpass.}
   \end{cases}
\end{equation}

If a peripheral arc occurs as part of a path, the homotopy type of the path will determine the choice of $\beta_{i,R}$ and $\beta_{i,L}$.
 As a result, we will often write just $\beta_i$.

A path $\alpha$ along the top of the knot can be written as a sequence of peripheral arcs (left or right for an edge), and crossing arcs. If the $i^{th}$ edge is level, the path along the top of the knot contains no crossing arc between two peripheral arcs, say $\beta_{i-1}$ and $\beta_i$.  We then write $\beta_{i} \gamma^0 \beta_{i-1}$  for this portion of the path, where $\gamma$ is the crossing arc based at $b_i$.  This is depicted in Fig.\ref{TwoOverpasses}.  Algebraically we also can consider this as a single peripheral path $\beta = \beta_i \beta_{i-1}$ along the concatenation of $E_i$ and $E_{i-1}$.

A path on top of the knot can contain a meridian (as one of the constituent peripheral arcs) if it follows a peripheral arc situated on an underpass, then hops over a crossing as a meridianal loop, and then comes back to the same underpass, continuing to its next peripheral arc. Such a situation is depicted in Fig. \ref{MeridianPath}, where the edges that form the underpass are labeled by $E_i, E_j$, and the adjacent overpass has an edge labeled by $E_k$. This fragment of the path can be written as $\beta_j \gamma \mu_k \gamma \beta_i$, where $\gamma$ is the crossing arc.

It follows that any path $\alpha$ along the top of a knot can be written as an alternating  sequence of peripheral arcs and crossing arcs.

Note that if $\alpha$ is a loop traveling around a region, then $\beta_i$ cannot be a meridian for any $i$, since any meridian would correspond to  Fig. \ref{MeridianPath} which does not occur here.

 \begin{figure}[h]
\centering
\begin{subfigure}[b]{0.53 \textwidth}
\centering
\includegraphics[scale=0.87]{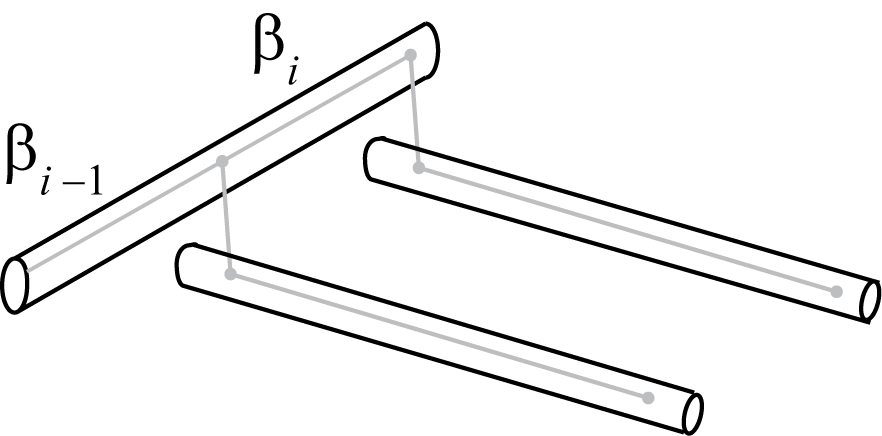}
\caption{Non-alternating link fragment.}\label{TwoOverpasses}
\end{subfigure}
\begin{subfigure}[b]{0.46 \textwidth}
\centering
\includegraphics[scale=1.2]{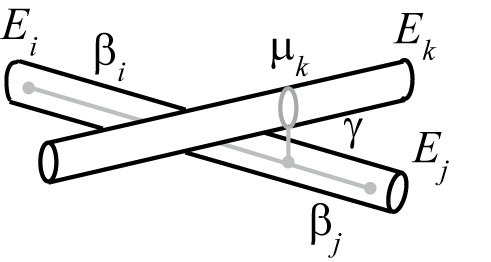}
\caption{Meridianal peripheral path}\label{MeridianPath}
\end{subfigure}
\caption{}
\label{Passes1}
\end{figure}

The following proposition summarizes what we proved in Sections 3 and 4. We will later use in the algorithm.

\begin{prop}\label{PropFormofpathalongtop} Any geometric representation  can be extended to preimages of peripheral arcs,  crossing arcs, and paths $\alpha$ on top of the knot.
In particular, let $\alpha$ be a path in a diagram $D$ homotopic to a sequence of peripheral and crossing arcs, and let $\tilde{\alpha}$ be a preferred lift of $\alpha$ in $H$.
For any such path,

\begin{equation}\label{eq1}\alpha = \gamma_{m_i} \beta^{\pm1}_{m_i} \dots \gamma_1 \beta^{\pm1}_{1}, \end{equation} and
\begin{equation}\label{eq2} \rho(\tilde{\alpha}) = \rho(\tilde{\gamma}_{m_i})  \rho(\tilde{\beta}^{\pm1}_{m_i}) \dots \rho(\tilde{\gamma}_1) \rho(\tilde{\beta}^{\pm1}_{1}), \end{equation}

\noindent where each ${\beta}_{m_i}=\beta_j$ is a peripheral arc on some edge $E_j$ (either left or right arc, or a simple meridian), each $\tilde{\gamma}_{m_i}$ is a  crossing arc, the exponents are uniquely determined by $\alpha$ and knot orientation,
and the endpoints of the lifts of arcs are chosen to coincide so that $\tilde{\alpha}$ is a path in $H$. 
\end{prop}

 \begin{proof} The equation (\ref{eq1}) summarizes the above discussions. To prove the rest, including equation (\ref{eq2}), note that by Lemma~\ref{lemma:extendtoperipheral}, a geometric representation $\rho$ determines a unique element of $\PSL_2(\C)$ corresponding to the chosen lift of $\beta_i$ (in the parabolic case, when the meridianal direction is chosen). Similarly, if $\gamma_i$ is a crossing arc, then Lemma~\ref{LemmaCrossinggeodesics} shows that after specifying the specific lift of $\gamma_i$ in $H$, this determines a unique element in $\PSL_2(\C)$ as well.
 \end{proof}

\section{Normalizing a Representation: Nice Matrices and Relations}\label{Normalizing}


In this section we show that elements of geometric representations have a ``nice" normalization. That is, up to conjugation we can write these elements in a prescribed way that captures the geometry of the corresponding path. In particular, we show that arcs, path and Wirtinger generators in a knot complement described in Proposition~\ref{PropFormofpathalongtop}  correspond to specific elements of $\PSL_2(\C)$, in matrix form. This is subsections \ref{normalizationofarcs} for arcs, and subsection \ref{normalizationofpaths} for paths. Moreover, we show that these``nice" elements of $\PSL_2(\C)$ satisfy simple relations. This is subsection \ref{simplerelations}.

In the next section, we will prove the reverse implication: we will show that we can define a representation using such normalized elements, assuming a few conditions are satisfied. Finally, in later sections, we will use all of this to outline the algorithm for finding character variety directly from a knot diagram.

\subsection{Nice Matrices for Peripheral and Crossing Arcs}\label{normalizationofarcs}

To represent a  Wirtinger generator $w_i$ of $\pi_1(S^3-K)$ up to homotopy, take the respective loop  $w_i=\alpha_i^{-1} \mu_i \alpha_i$ where  $\alpha_i$ is a path on top of the knot from $b$ to $b_i$. As Wirtinger generators are conjugate to the standard meridian, we must have $\rho(w_i)=\Ax \Mx \Ax^{-1}$ for some $\Ax\in \PSL_2(\C)$, with $\Mx$ described in Definition~\ref{definition:meridian} and determined in Lemma~\ref{lemma:meridian}. 

By Remark~\ref{Mbanana} and Lemma~\ref{lemma:meridian},  the matrix $\Mx$ corresponds to the conjugate of $\rho(\tilde{\mu})$ where $\tilde{\mu}$ lies on a banana $B$ with  endpoints $i_0$ and $\infty$.  Using lemma \ref{lemma:extendtoperipheral}, for a peripheral arc $\beta$,  there is a  conjugate of $\rho(\beta)$  corresponding to  $\rho(\tilde{\beta})$ for $\tilde{\beta}$ on $B$ which  will always be upper triangular and commute with $\Mx$, since $\tilde{\beta}$ and $\tilde{\mu}$ lie on the same banana.  With this in mind we have the following definitions.

\begin{definition}\label{definition:PxandCx}
 Let an \underline{edge matrix} $\Px_i$ be the conjugate of $\rho(\tilde{\beta}_i)$ that is upper triangular; $\Px_i$ is of the form $(\pm) \mat{v_i}{u_i}{0}{v_i^{-1}}$.  Let a \underline{crossing matrix} $\Cx_i$ be the conjugate of $\rho(\tilde{\gamma}_i)$ that is of the form $(\pm) \mat{0}{c_i}{-c_i^{-1}}{0}$ as given in Lemma~\ref{LemmaCrossinggeodesics}. Here and further, the $\pm$ sign before a matrix refers to the fact that we are working in $\PSL(2, \mathbb{C})$, i.e. actually with the equivalence classes of matrices. We will also reserve a notation $\mathcal{W}_i$ for the alternative form of the crossing matrix: $\mathcal{W}_i=(\pm) \mat{0}{-w_i}{1}{0}$. With $w_i =c_i^2$ we then have $\Wx_i = -c_i  \Cx_i$.
\end{definition}

\begin{remark}\label{orientation}Note that if $\Px_i$ is associated to a peripheral arc  $\beta_i$, and the arc $\beta_i$ appears in some path $\alpha$ with the orientation opposite to the orientation of the edge of $D$ where $\beta_i$ is (i.e. it appears as $\beta_i^{-1}$), then we substitute $\Px_i$ with $\Px_i^{-1}$. In what follows, we will therefore often use $\Px_i^{\pm1}$, referring to this context.
\end{remark}

For a crossing arc $\gamma_i$, the associated $\Cx_i$ from Definition~\ref{definition:PxandCx} is uniquely determined by Lemma~\ref{LemmaCrossinggeodesics}.  In the next lemma, we will establish that a similar fact holds for a peripheral arc $\beta_i$ and the associated matrix $\Px_i^{\pm1}$.

A matrix  $\Px_i^{\pm1}$ commutes with $\Mx$ exactly when  the fixed points of $ \Px_i^{\pm1}$ and the fixed points of $\Mx$ to coincide, since then they have the same axis. This set up corresponds to conjugating so that a meridian and a chosen peripheral arc are on the same preferred banana.

 \begin{definition}\label{def:PandC}
We call the matrix equation $(\pm) \Px_i\Mx = (\pm) \Mx\Px_i$, or the equivalent equation  $(m-m^{-1})u_i=(\pm)(v_i-v_i^{-1})$   the \underline{commuting equation}. Recall that by Definition~\ref{definition:meridian}, $m$ is the $(1,1)$ entry of $\rho(\mu)=\Mx$. If we have the matrix $\Px_i^{-1}$ instead, as in Remark \ref{orientation}, the commuting equation is  $(\pm) \Px_i^{-1}\Mx= (\pm) \Mx\Px_i^{-1}$ with  $(m-m^{-1})u_i=(\pm)(v_i-v_i^{-1})$ unchanged. 
\end{definition}

\begin{remark}\label{remark:commutingsign}
In practice (e.g in the algorithm that will follow), will use the equation $(m-m^{-1})u_i= (v_i-v_i^{-1})$ which amounts to assigning a sign to the matrices $\Px_i$. In Remark~\ref{remark:liftingtoSL}, we discuss how this affects signs in the representation as a whole.  Often it is useful to eliminate the $u_i$ variables using $u_i= (v_i-v_i^{-1})/(m-m^{-1})$. This substitution does not work for parabolic representations because the commuting equation is then trivial.
\end{remark}

\begin{lemma}\label{lemma:Pmatrixunique}
Let $\beta_i$ be a peripheral arc. The matrix $\Px_i^{\pm1}$ associated to $\beta_i$ as in Definition~\ref{definition:PxandCx} is uniquely determined given that $\Px_i^{\pm1}$ satisfies the commuting equation.
\end{lemma}

\begin{proof}

 For a peripheral arc $\beta_i$, the trace of $\rho(\beta_i)$ does not depend on the chosen preimage of $\beta_i$, as it is invariant under conjugation. This fact together with the commuting equation uniquely define the matrix $\Px_i^{\pm1}$ associated to a peripheral arc $\beta_i$.\end{proof}

\subsection{Normalizing Paths: Writing Them In Terms Of Nice Matrices}\label{normalizationofpaths}

Above, we always used natural labeling. But given a path $\alpha$, we will at times use numerical subscripts for arcs that occur in $\alpha$ to indicate their position in $\alpha$, therefore differing from the natural labeling. We will explicitly note when this is the case.

The next lemma allows us to rewrite any path on top of the knot, and indeed any Wirtinger generator as a product of matrices $\Mx$, $\Px_i^{\pm1}$, and $\Cx_i$ for various indices $i$. Specifically, any such path lifts to a series of paths on bananas and geodesic paths connecting bananas.  These are each conjugate to a specific $\Px_i^{\pm1}$ or $\Cx_i$.  The conjugation of two adjacent paths of this type to $\Px_i^{\pm1}$ and $\Cx_i$ can be simplified because we can cancel most of the terms in the conjugating matrices (since the paths from the bananas to the preferred banana for two adjacent paths of this type are almost the same).  The following lemma makes this explicit.

\begin{lemma}\label{lemma:conjuationbacktracking}
 Let $\tau = \tau_n \tau_{n-1} \dots \tau_1$ be a path in $M$ consisting of alternating peripheral and crossing arcs, with indices corresponding to the order of arcs in the path $\tau$. Fix banana $B_1$, and a point $z_2\in \partial \mathbb{H}$.  Let $B_2$ be a banana with preferred ideal point at $z_2$.
 Let $\tilde{\tau}$ be a preferred lift of $\tau$ with the initial point $\tilde{t}$ on  $B_1$. Assume that the transformation corresponding to $\rho(\tilde{\tau}_i)$ starts at $\tilde{t}_n$. Then for any $i$, we may choose the conjugate of $\rho(\tilde{\tau}_i)$ with the initial point on $B_1$ instead: if $\tau_i$ is a crossing arc, assume that the geodesic representative for  $\tilde{\tau}_i$ extends to $z_2$. This conjugate will be denoted by $T_i$. After choosing such conjugate for every $i$, the path $\tau$ is represented as 
  \[ \rho(\tilde{\tau}) =  \rho( \tilde{\tau}_n ) \rho(\tilde{ \tau}_{n-1}) \dots \rho(\tilde{ \tau}_1) = T_1 T_2 \dots T_n.\]

\end{lemma}

\begin{proof}
 Let $t_i$ and $t_{i+1}$ denote the initial and terminal point of $\tau_i$ respectively, and let $\tilde{t_i}$ and $\tilde{t_{i+1}}$ be such points for $\tilde{\tau}_i$. 
If $n=1$, $\rho(\tilde{\tau})=\rho(\tilde{\tau_1})=T_1$ since the basepoint of $\tau$ is already on $B_1$, and no conjugation for $\tau_1$ is needed.  If $n=2$, then we have $\rho(\tilde{\tau})= \rho(\tilde{\tau}_2) \rho(\tilde{\tau}_1)$ where we are mindful that these are the representations corresponding to the lifts of $\tau_1$ at $\tilde{t}$ and of $\tau_2$ at $\tilde{t}_2$.   Then $\rho(\tilde{\tau}_1)=T_1$ and $\rho(\tilde{\tau}_1)^{-1} \rho(\tilde{\tau}_2)\rho(\tilde{\tau}_1) = T_2$.  (To check, this conjugate takes $\tilde{t}$ to $\tilde{t}_1$.)  Therefore
\[
\rho(\tilde{\tau})= \rho(\tilde{\tau}_2) \rho(\tilde{\tau}_1) = \rho(\tilde{\tau}_1)  T_2 \rho(\tilde{\tau}_1)^{-1} \rho(\tilde{\tau}_1) =   \rho(\tilde{\tau}_1)  T_2 =T_1T_2
\]
since $\rho(\tilde{\tau}_1)=T_1$.

Now we argue inductively.   We assume that
\[ \rho(\tau_n \tau_{n-1} \dots \tau_1) = \rho(\tilde{\tau}_n) \big(  \rho(\tilde{\tau}_{n-1}) \dots \rho(\tilde{\tau}_2) \rho(\tilde{\tau}_1) \big).\]
The transformation corresponding to $\rho(\tilde{\tau}_n)$ is the conjugate of $T_n$ that starts at $\tilde{t}_n$ so that we have
\[
\rho(\tilde{\tau}_n) = \big(  \rho(\tau_{n-1}) \dots \rho(\tilde{\tau}_1)        \big) T_n \big(    \rho(\tilde{\tau}_1)^{-1} \dots  \rho(\tilde{\tau}_{n-1})^{-1}                 \big).
\]
This can be checked by noting that $   \rho(\tilde{\tau}_1)^{-1} \dots  \rho(\tilde{\tau}_{n-1})^{-1} (\tilde{t}_{ n })=\tilde{t}_1$  since
$\rho(\tilde{\tau}_i)(\tilde{t}_i) = \tilde{t}_{i+1}$.  Therefore,
\begin{align*}
  \rho(\tau_n \tau_{n-1} \dots \tau_1)
  & = \big(  \rho(\tau_{n-1}) \dots \rho(\tilde{\tau}_1)        \big) T_n \big(    \rho(\tilde{\tau}_1)^{-1} \dots  \rho(\tilde{\tau}_{n-1})^{-1}                 \big) \\
  & \qquad \qquad  \qquad \qquad \qquad \qquad  \big(  \rho(\tilde{\tau}_{n-1}) \dots \rho(\tilde{\tau}_2) \rho(\tilde{\tau}_1) \big) \\
  & = \big(  \rho(\tau_{n-1}) \dots \rho(\tilde{\tau}_1)        \big) T_n.
  \end{align*}
Inductively, we have that $ \rho(\tau_{n-1}) \dots \rho(\tilde{\tau}_1) = T_1 \dots T_{n-1}$. Therefore, \[\rho(\tau_n \tau_{n-1} \dots \tau_1) =T_1 \dots T_n.\]
\end{proof}

\begin{remark}\label{WirtingerInMatrices}
By Lemma~\ref{lemma:conjuationbacktracking} and using the notation from Definition~\ref{definition:PxandCx}, for each path $\alpha$ along the top of the knot from the base point $b$ to the $k^{th}$ edge, we have the matrix
 $ \Ax= \Px_{i_1}^{e_1}  \Cx_{i_1}^{f_1}  \dots \Px_{i_k}^{e_k} \Cx_{i_k}^{f_k} $ that coincides with the path in reverse order. Here each of $\Px_{i_j}$ (or $\Cx_{i_j}$) is one of the $\Px_i$ (or $\Cx_i$) for $i\in\{1,2,\dots 2n\}$ according to the natural labeling. The exponent conventions reflect the direction of the peripheral arcs and allow for non-alternating diagrams, as discussed in subsection \ref{definingrepresentations}: each $e_i, i=1, ..., k,$ is $+1$ if the direction of the path agrees with the direction of $K$, and $-1$ otherwise. If the path goes across an overpass, that is,  from one edge to the next where both crossings are overpasses (without traversing the crossing arc) then $f_j=0$.  (See Figure~\ref{Passes1} (1) for a picture of this situation.)   Otherwise, $f_j=1$.
In the case where the path goes from one edge to the next by leaping over an overpassing arc (that is, it continues on the continuation of the previous edge) then we write this as $\Cx \Mx^{\pm 1} \Cx$ where the sign corresponds to the direction of the meridional crossing.  (See Figure~\ref{Passes1} (2) for a picture of this situation.)  With this,
$ \rho(w_i) = \Ax \Mx \Ax^{-1}$. \end{remark}

\begin{lemma}\label{independence}
 With the above notation, for a Wirtinger generator $w_i$, $\rho(w_i)$ does not depend on the choice of path on the top of the knot $\alpha$.  
\end{lemma}

\begin{proof}

The independence of $\alpha$ follows from the uniqueness of $\Px_i$ and $\Cx_i$ associated to every arc, and the fact that a different choice of $\alpha$ will result in a homotopic path.
\end{proof}


\subsection{Simple Relations from the Necessary Conditions}\label{simplerelations}

A  Wirtinger generator $w_i$ is homotopic rel endpoints to a loop of the form $\alpha^{-1} \mu_i \alpha$ where $\alpha=\gamma_{n} \beta_n \dots \gamma_1\beta_1$ is a path on top of the knot consisting of peripheral and crossing arcs (suppressing exponents). Here the indices correspond to the order of arcs in $\alpha$.

Recall that by construction (see Equation~\ref{EquationEdges1}), the ``left'' and ``right'' peripheral arcs $\beta_{i,R}$ and $\beta_{i,L}$ for an edge $E_i$ are related by  $\beta_{i,L} = \mu^{g_i} \beta_{i,R}$. Here the edge $E_i$ inherits the orientation of $K$, and $g_i$ satisfies \ref{EquationEdgesExp}  according to this orientation.
This  implies that for a fixed lift, including a lift with base point at $B_1$ we have $\tilde{\beta}_{i,L} = \widetilde{  \mu^{g_i} \beta_{i,R}}$ so that $\rho( \tilde{\beta}_{i,L} ) =  \rho( \widetilde{  \mu^{g_i} \beta_{i,R}})= \rho(\tilde{\mu}^{g_i}) \rho(\tilde{\beta}_{i,R})$ by Proposition \ref{PropFormofpathalongtop}. By Lemma~\ref{lemma:conjuationbacktracking}, $\Px_{i,L} = \Px_{i,R} \Mx^{g_i}$. Since  $\Px_{i,R}$ and $\Mx$ are transformations corresponding to the lifts of peripheral arcs on the same banana, they satisfy commuting equation (Definition \ref{def:PandC}), and $\Px_{i,L}=\Mx^{g_i}\Px_{i,R}$.

 Let $\alpha$ be a loop that goes around a region in $D$, the taut diagram of our knot.  Such a loop $\alpha$ is null-homotopic. Then $\alpha$ is homotopic to  a  loop of the form $\gamma_{n} \beta_n \dots \gamma_1\beta_1$  where the $\beta_i$ are peripheral arcs and the $\gamma_i$ are crossing arcs indexed according to their order in $\alpha$.  By Proposition~\ref{PropFormofpathalongtop} and Lemma~\ref{lemma:conjuationbacktracking}, up to conjugation we have 
\[ \rho(\alpha) = \Ax= \Px_1^{\pm1}\Cx_1 \dots \Px_n^{\pm1} \Cx_n. \]
Since $\alpha$ is homotopically trivial,  $\rho(\alpha) = I$.

 \begin{definition}\label{eregiondgeequations1} We call a matrix relation of the form    
$\Px_{i,L}  = \Mx^{g_i}\Px_{i,R}$ an  \underline{edge} \underline{equation}. 
 \end{definition}

 \begin{definition}\label{eregiondgeequations2} Let $R$ be any region in $D$ and $\alpha$ a path around $R$. Suppose $\rho(\alpha)$ is conjugate to $\Ax=\Px_{i_1}^{e_1}  \Cx_{i_1}^{f_1}  \dots \Px_{i_k}^{e_k} \Cx_{i_k}^{f_k}$ as in Remark \ref{WirtingerInMatrices}, where $\Px_i=\Px_{i, L}$ if that is the arc in the interior of the region $R$, and $\Px_i=\Px_{i, R}$ otherwise, and the exponents are chosen according to Remark \ref{orientation}. We call the resulting matrix relation $\Px_{i_1}^{e_1}  \Cx_{i_1}^{f_1}  \dots \Px_{i_k}^{e_k} \Cx_{i_k}^{f_k}=I$ the \underline{region equation} for $R$.

\end{definition}

With this notation, edge and region equations are satisfied for all geometric representations due to the geometric observations and definitions above.

For a given region, there are many equivalent ways of writing the region equation by cyclically permuting the starting point, or traversing a given region clockwise or counterclockwise.  All are algebraically equivalent.

\section{The Other Direction: Normalized Matrices and Relations Define Representations}\label{section:wirtinger}

We now show that if we simple-mindedly define $\Mx$, $\Px_i$, and $\Cx_i$ matrices to correspond to the meridian, peripheral paths, and crossing paths that so long as they satisfy the edge, region, and commuting conditions, this determines a representation.

\begin{prop}\label{prop:backwards}
Let $K$ be a knot with an oriented taut diagram $D$. Assign a meridian matrix $\Mx$ to the meridian, an edge matrix to each oriented peripheral arc (i.e. to each side of every edge), and a crossing matrix  to each crossing arc in $D$ so that the commuting, edge, and region equations are satisfied.  For each Wirtinger generator $w_n$ with basepoint $b_n$, let $\alpha_n$ be a path on top of the knot from $b$ to $b_n$ of the form $\alpha_n = \gamma_n \beta_n \dots \beta_1 \gamma_1$, with indices corresponding to the order of arcs in $\alpha$ (where in the situation depicted in Fig.\ref{TwoOverpasses} we take the associated $\gamma_i$ to be the identity).   Then  setting  $\rho(w_n)=\Ax_n \Mx\Ax_n^{-1}$, where $\Ax_i=\Px_1^{\pm1}\Cx_1 \cdots \Px^{\pm1}_{n} \Cx_n$, with $\Px_i^{\pm1}$ corresponding to $\beta_i$, and $\Cx_i$ corresponding to $\gamma_i$, defines a representation $\rho$ of $\pi_1(M)$ to $\PSL_2(\C)$. \end{prop}

 \begin{proof}
It is enough to show that $\rho$  is well-defined, and satisfies the  Wirtinger relations.     The region and edge equations and Proposition \ref{PropFormofpathalongtop} imply that  for any path $\alpha_n$ as above, the matrix for $\rho(\alpha_n)$ depends only on the initial and terminal point of $\alpha_n$, and is well-defined depending only on the homotopy class of $\alpha_n$.

We now prove that Wirtinger relations hold.  Depending on the orientation of the knot $K$, there are two cases.  So consider one of the cases: Figure \ref{ij}, left, shows a labeling of edges of a crossing in a knot diagram, with orientation. Fix a basepoint $b$ in $S^3-K$. Let $w_m$ be a Wirtinger generator starting at $b$, and wrapped around the edge $m$ on the figure. We need to show that $\rho(w_j)\rho(w_i)=\rho(w_{i+1})\rho(w_j)$. Denote by $\Ax_m$ the path on top of the knot from $b$ to the edge labelled $m$, where $m=i, i+1$ or $j$.

We have $\rho(w_i) = \Ax_i  \Mx \Ax_i^{-1}$, $\rho(w_j) = \Ax_j\Mx \Ax_j^{-1} $, and $\rho(w_{i+1})=\Ax_{i+1}\Mx \Ax_{i+1}^{-1} $.  We can take $\Ax_j= \Ax_i\Px_i \Cx_i$ and $\Ax_{i+1}=\Ax_i\Px_i  \Cx_i \Mx \Cx_i$, and, without loss of generality, we may assume that the exponent of $\Px_i$ here is positive. Therefore,
\begin{align*}
 \rho(w_j)\rho(w_i) & = \big(( \Ax_i\Px_i \Cx_i) \Mx ( \Ax_i\Px_i \Cx_i)^{-1}      \big) \big(  \Ax_i  \Mx \Ax_i^{-1}   \big) \\
 & =  \Ax_i \Px_i \Cx_i \Mx \Cx_i \Px_i^{-1} \Mx \Ax_i^{-1}  \\
 \rho(w_{i+1})\rho(w_j)  & =\big( (\Ax_i\Px_i  \Cx_i \Mx \Cx_i ) \Mx (\Ax_i\Px_i  \Cx_i \Mx \Cx_i )^{-1} \big)  \big(( \Ax_i\Px_i \Cx_i) \Mx ( \Ax_i\Px_i \Cx_i)^{-1}      \big)      \\
&  = \Ax_i \Px_i \Cx_i \Mx \Cx_i \Mx  \Px_i^{-1} \Ax_i^{-1} .
\end{align*}
These are equal since $\Mx$ and $\Px_i$ commute, and hence the Wirtinger relation holds.

The proof for the other case, with different orientation of the link at the crossing (as in Figure \ref{ij}, right) is similar.
\end{proof}

 \begin{figure}[H]
\centering
\includegraphics[scale=0.7]{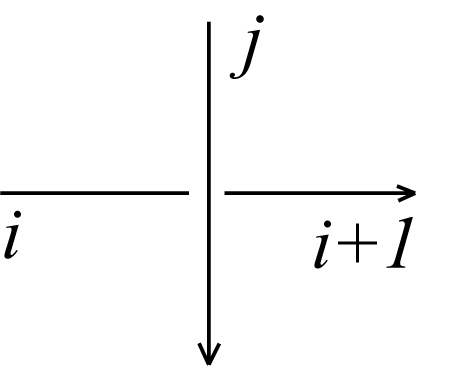} \hspace{1in}
\includegraphics[scale=0.7]{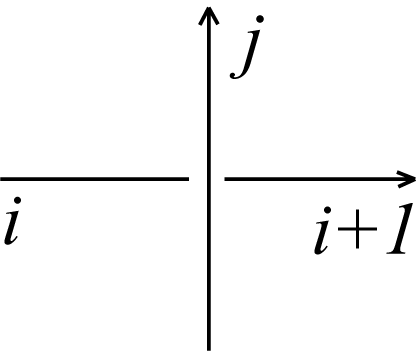}

\caption{Two cases for the orientation of the crossing.}
\label{ij}

\end{figure}

We have now proven the main result of our paper.

\begin{thm}\label{ThmMain1}
Let $K$ be an oriented knot with a taut diagram $D$, and $\rho$ be a geometric representation of $\pi_1(S^3-K)$ to $\text{PSL}_2(\C)$. Then up to conjugation, $\rho$ determines the meridian matrix $\Mx$ (as in Definition \ref{definition:meridian}), and edge and crossing matrices
$\Cx_i$, and $\Px_i$ (as in Definition \ref{definition:PxandCx}) for the preferred meridian, crossing arcs, and oriented peripheral arcs. These matrices are unique given the meridianal direction unless $\mu$ lifts to an order 2 elliptic element. Each $\Px_i$ satisfies the commuting equation, and the region and edge equations are satisfied.

Conversely,  given matrices $\Mx$, $\Cx_i$, and $\Px_i$ as above, satisfying the region equations, edge equations, and commuting equations, the following holds. By defining $ \rho(w_n) = \Ax_n \Mx \Ax_n^{-1} $, where $\Ax_n$ corresponds to a path $\alpha_n$ along the top of the knot, we determine a representation of $\pi_1(S^3-K)$ to $\mathrm{PSL}_2(\C)$. Unless $\mathrm{trace}(\Mx)=0$, the representation $\rho$ is unique up to conjugation.

\end{thm}

\begin{cor}\label{corollary:geometric}
 Assume that $S^3-K$ is hyperbolic. Given matrices $\Mx$, $\Cx_i$, and $\Px_i$  satisfying the region equations, edge equations, and commuting equations, the set of all  $\rho(w_n)$   contains all representations up to conjugation in a canonical component of the $\PSL_2(\C)$ representation variety for $\pi_1(S^3-K)$.  
\end{cor}

\begin{proof} 
If the knot complement is hyperbolic, then all but finitely many Dehn fillings are hyperbolic and these correspond to geometric representations (see Section~\ref{section:Zariski}). Therefore, since there are infinitely many such fillings, and the complex dimension of the canonical component of the character variety is one, Theorem~\ref{ThmMain1} determines (at least) the canonical component of the character variety. 
\end{proof}

 \section{General Algorithm to Determine  Geometric Representations}\label{section:algorithm}

 Following Theorem~\ref{ThmMain1}, we explicitly state the algorithm that gives equations for components of the representation variety (up to conjugation), including the canonical component.  Assume that $D$ is an oriented taut diagram for a knot $K$, with a base point $b$ on the knot.

\begin{alg}\label{algorithm:main}

\begin{enumerate}

\

 \noindent\makebox[\linewidth]{\rule{5.4in}{0.4pt}}
\item[] \textbf{Step 1. Labelling the knot diagram}
\vspace{-0.1in}

\noindent\makebox[\linewidth]{\rule{5.4in}{0.4pt}}

\noindent {\bf 1a)}  Label the meridian based at $b$ with $\Mx=(\pm) \mat{m}{1}{0}{m^{-1}}$.

\noindent {\bf 1b)} Orient all edges compatibly with the orientation of $K$, and label according to the natural labeling. For an oriented edge $E_i$, label each peripheral arc, left and right, with the matrices $\Px_{i, L}=  (\pm) \mat{v_{i, L}}{u_{i, L}}{0}{v_{i, L}^{-1}}$ and $\Px_{i, R}=  (\pm) \mat{v_{i, R}}{u_{i, R}}{0}{v_{i, R}^{-1}}$ respectively.

\noindent {\bf 1c)} Label each crossing with the matrix $\Cx_j = (\pm) \mat{0}{c_j}{-c_j^{-1}}{0}$ with indices corresponding to the natural labeling.

 \noindent\makebox[\linewidth]{\rule{5.4in}{0.4pt}}

\item[] {\bf Step 2. Writing down the equations}

\vspace{-0.1in}

\noindent\makebox[\linewidth]{\rule{5.4in}{0.4pt}}

\noindent {\bf 2a)} For each edge matrix $\Px_i$, the commuting equation  from Definition~\ref{def:PandC}) holds. (In practice, we assign a matrix in $\SL_2(\C)$ for each edge, as opposed to a coset in $\PSL_2(\C)$.)

\noindent {\bf 2b)} For each peripheral arc, the edge equation holds as in Definition \ref{eregiondgeequations1}: $ \Px_{i, L} = (\pm) \Mx^{g_i} \Px_{i, R}.$ (See  Equation~\ref{EquationEdgesExp} for the conventions defining $g_i$.) 
This is equivalent to
$ v_{i, L}v_{i, R}^{-1}=m^{g_i}, \ \mathrm{ and } \  u_{i, L}v_{i, R}-u_{i, R}v_{i, L}=g_i.$

\noindent {\bf 2c)} For every region of $D$, the region equation holds as in Definition \ref{eregiondgeequations2}.

 \noindent\makebox[\linewidth]{\rule{5.4in}{0.4pt}}

\item[] {\bf Step 3. Defining Wirtinger Generators}

\vspace{-0.1in}

\noindent\makebox[\linewidth]{\rule{5.4in}{0.4pt}}

\noindent {\bf 3)} Let $\alpha_i$ be a path along the top of the knot from $b$ to the $i^{th}$ peripheral arc such that an associated Wirtinger generator is $w_i=\alpha_i m \alpha_i^{-1}$. Associate to $\alpha_i$ a matrix  $\Ax_i$ as in Remark \ref{WirtingerInMatrices}. Then 
$\rho(w_i) =   \Ax_i  \Mx \Ax_i^{-1} .$

\end{enumerate}
\end{alg}

\begin{remark}\label{shortcuts} Shortcuts. \begin{enumerate} 
\item Steps 1 and 2 already produce the equations that determine the canonical component of the representation variety by  Corollary~\ref{corollary:geometric}. It is however often useful to have explicit matrices for Wirtinger generators for a representation, and this is achieved in Step 3.
\item  To reduce the number of matrix labels, one can choose to label only one side of each edge, either left or right, in Step 1b. The label for the other side is then easily determined by the edge equation (Definition \ref{eregiondgeequations1}). One practical way to do it is to color the regions of the knot diagram in black and white, as a checkerboard, choose explicit edge matrices in regions of one color, say black, and then use the edge equations to write the edge matrices for white regions. For example, if an edge $e$ in a black region is labeled $\Px_i$, then in a white region $e$ is labeled $\Px_i \Mx^{\pm1}$.

Alternatively, one can label every peripheral arc, left and right, with a new matrix, but eliminate some of the new matrix elements, i.e. $u_i$. Indeed, up to conjugation, there are only  finitely many representations with $m=\pm 1$ on a canonical component. Therefore there are infinitely many representations so that $m-m^{-1} \neq 0$, which is a Zariski dense set. For these representations, using the commuting equations (Definition~\ref{def:PandC}), we can set  $u_i=(v_i-v_i^{-1})/(m-m^{-1})$.  Therefore, we can make this substitution for all representations on a canonical component.
\item Each region equation in Step 2c  yields at most three independent polynomial equations. (The   determinant 1 condition makes at least one of the four equations dependent on the others.)
\item In Definition \ref{definition:PxandCx}, we specified that we work with equivalence classes of matrices when we label a knot diagram, as in Step 1. One can use just matrices instead, as long as the region equations equal $(\pm)$ the identity.

Indeed, the matrices of type $\Px$ and $\Cx$ appear in products forming the $\Ax$ matrices, which conjugate $\Mx$ to form the Wirtinger generators.  As such, any choice of sign of $\Px$ or $\Cx$ does not affect the representation.  Choosing a sign for $\Px_R$ and $\Mx$ determines a sign for $\Px_L$ by the edge equation (Definition~\ref{eregiondgeequations1}).  Any choice of sign for  $\Px$ and $\Mx$ will not affect the commuting equation (Definition~\ref{def:PandC}).

\item We can substitute $\Cx$ by $\mathcal{W}$ in the algorithm as follows. Using Definition \ref{definition:PxandCx}, we can write:

\[ \Ax =  \pm \Px_1 \Cx_1 \dots \Px_k\Cx_k = \pm (c_1\dots c_k)   \Px_1 \mathcal{W}_1 \dots \Px_k \mathcal{W}_k, \] 
with indices corresponding to the order of matrices in $\Ax$. A region equation $\Ax= I$ is then equivalent to
\[  \Px_1 \mathcal{W}_1 \dots \Px_k \mathcal{W}_k = xI \] for some non-zero complex number $x$. Alternatively,  one can require the $(1,2)$ and $(2,1)$ entries to equal zero and the $(1,1)$ and $(2,2)$ entries to be equal.  This was used in \cite{MR3190595}.

\end{enumerate}
\end{remark}

\begin{remark}
We make a few choices  so that  representations are unique up to conjugation. The first is our choice of $\Mx$ as upper triangular with a 1 on the off diagonal as in Lemma~\ref{lemma:meridian}.  Second, we can choose an adjacent Wirtinger generator to be sent to a lower triangular matrix as in Proposition~\ref{PropositionConjugatingmeridianandwirt}. These choices are unique except for (the finitely many) representations with $\tr \rho (\mu)=0$. 
\end{remark}

\begin{thm}\label{ThmAlg}
Any  representation satisfying Algorithm~\ref{algorithm:main} is  a geometric representation. Conversely, all geometric representations satisfy the conditions of Algorithm~\ref{algorithm:main}. 
\end{thm}

\begin{proof}

Assume that $\rho:\pi_1(M) \rightarrow \text{PSL}_2(\C)$ is geometric. We have shown that $\rho$ extends to crossing and peripheral arcs in Proposition~\ref{PropFormofpathalongtop}. We have also shown that $\rho$ has a unique $\text{PSL}_2(\C)$ representative. Indeed, for each peripheral arc $\beta$, $\rho(\tilde{\beta})$ is conjugate to an edge matrix $\Px$, and for each crossing arc $\gamma$, $\rho(\tilde{\gamma})$ is conjugate to a crossing matrix $\Cx$ given in Definition~\ref{definition:PxandCx}.   The matrices $\rho(\beta)$ and $\rho(\gamma)$  are uniquely defined for a specific lift of an arc by Lemma~\ref{lemma:extendtoperipheral} and Theorem~\ref{Thm:crossingarc}, respectively.  The edge matrices satisfy  the commuting equation by Lemma~\ref{lemma:Pmatrixunique}.  Then    the uniqueness of the element $\Px$ for a peripheral arc follows from Lemma~\ref{lemma:Pmatrixunique}. Lemma~\ref{LemmaCrossinggeodesics} shows that an element $\Cx$ is unique for a crossing arcs.  Therefore the matrices $\Px$ and $\Cx$ are uniquely determined, depending only on the orientation of the knot.  Proposition~\ref{prop:backwards} demonstrates that we can write a path on the top of the knot as a sequence of $\Px$ and $\Cx$ matrices.

Further we have shown that edge equations hold for each edge of the diagram (see Section~\ref{simplerelations}, Section 3 of \cite{MR3190595}, and equation \ref{EquationEdges1}). The region equations hold for each region of the diagram because the corresponding loops are null homotopic (see Section~\ref{simplerelations}).  The independence of $\rho(w_i)$ of the chosen path $\alpha$ follows from the region equations which ensure path independence. 

This proves that Algorithm \ref{algorithm:main} works.

The converse statement follows from Proposition~\ref{prop:backwards}.\end{proof}

\begin{remark}\label{remark:liftingtoSL}

 \textit{$\SL_2(\C)$ representations.} Algorithm \ref{algorithm:main} determines representations associated to the Wirtinger presentation of the knot group.  In this presentation, the generators of $\pi_1(M)$ are all meridians: that is, they are freely homotopic to $\mu$.
 For a preferred meridian $\mu$  and  $\rho:\pi_1(M)\rightarrow \PSL_2(\C)$, we have   $\rho(\mu)=\pm \Mx$ for  $\Mx\in \SL_2(\C)$. The two $\SL_2(\C)$ representations which are lifts of this $\PSL_2(\C)$ representation
  can be determined as follows.  The lifts are $\tilde{\rho}_1$ and $\tilde{\rho}_2$ by specifying $\tilde{\rho}_1(\mu)=\Mx$ and $\tilde{\rho}_2(\mu)=-\Mx$ (see Section~\ref{section:liftingreps}).  Moreover, if $w_i$ is a Wirtinger generator, and $\rho(w_i)= \pm \Ax_i^{-1}  \Mx \Ax_i$, we have $\tilde{\rho}_1(w_i) = \Ax_i^{-1}  \Mx \Ax_i$ and $\tilde{\rho}_2(w_i) = - \Ax_i^{-1}  \Mx \Ax_i$.

Let $\rho$ be a $\PSL_2(\C)$ representation without 2-torsion, so that $\rho$ lifts to an $\SL_2(\C)$ representation.  We can assign signs to the  $\Px$ and $\Cx$ type matrices so that they are in $\SL_2(\C)$. The signs of the $\Px$ and $\Cx$ matrices do not affect the lift of the $\PSL_2(\C)$ representation to $\SL_2(\C)$ as these matrices only appear in the $\Ax$ terms above and so any sign difference cancels out in a Wirtinger generator.  It is possible to assign signs in a way that all region equations (Definition~\ref{eregiondgeequations2}) equal the identity because an $\SL_2(\C)$ lift exists and the region equations represent loops.   Any choice of signs for $\Px$ matrices will allow us to use signed commuting equations, as mentioned in Remark~\ref{remark:commutingsign}.  The edge equations (Definition~\ref{eregiondgeequations1}), as they are based on the geometry of the associated transformations are satisfied with either a $-$ or a $+$. That is, we have either $\Px_{i,L}=\Mx^{g_i}\Px_{i,R}$ or  $\Px_{i,L}=- \Mx^{g_i}\Px_{i,R}$.

When solving equations in practice in both the $\SL_2(\C)$ and $\PSL_2(\C)$ case, it is often easier to assign signs to the $\Px$ and $\Cx$ matrices by using the commuting equations and the $+$ solution to the edge equations. Then one solves region equations as equaling $\pm I$, i.e. by ensuring the off-diagonal entries are zero and the diagonal entries are the same.  This will not affect the Wirtinger relations.  Instead, when realizing the longitude $\lambda$ as the concatenation of $\Px$ matrices, the corresponding matrix may be $-\rho(\lambda)$ due to the sign choice. The difference between the $\SL_2(\C)$ and $\PSL_2(\C)$ representations is then as follows: the meridians are cosets in the $\PSL_2(\C)$ case, and there are two lifts in the $\SL_2(\C)$ case for a meridian, where the sign of each lift of the $\Mx$ matrix governs the signs of the other matrices. 

\end{remark}

\begin{remark}\label{links}
A modification of our algorithm will work for links. One important difference is that for different components of a link, the peripheral matrices will not commute. As a result, representing peripheral elements for different link components might become rather cumbersome.
\end{remark}

\section{Further Shortcuts: Bigons and 3-sided Regions}\label{section:smallregions}

\subsection{Bigon Regions}  

 We now show that in the special case when a region $S$ has just two edges, the  edge and crossing matrices simplify considerably.  This observation is particularly helpful for diagrams with twist regions.

 \begin{lemma}\label{lemma:bigon}
Given a knot diagram $D$, let $S$ be its region with exactly two edges.
Then the edge matrices inside $S$ are both $\pm I$, and the crossing matrices for $S$ are identical.
\end{lemma}

\begin{proof}
The region equation for $S$ is $\Px_1 \Cx_1 \Px_2 \Cx_2=\pm I$, i.e. 
\begin{center}
$\Px_1\Cx_1= \pm(\Px_2\Cx_2)^{-1} = \pm \Cx_2 \Px_2^{-1}$.
\end{center}

\noindent With $\Px_i$ and $\Cx_i$ as in Definition \ref{definition:PxandCx}, this equation is
\[
\mat{-u_1 c_1^{-1}}{v_1c_1}{-v_1^{-1}c_1^{-1}}{0} = \pm \mat{0}{c_2v_2}{-c_2^{-1}v_2^{-1}}{c_2^{-1}u_2}.
\]

\noindent From matrix entries, we conclude that since $c_i\neq 0$ (as $\Cx_i$ is not trivial), we have  $u_1=u_2=0$ and $c_1v_1=\pm c_2 v_2$.
 The commuting equation (Definition~\ref{def:PandC})  specifies that
\[ (m-m^{-1})u_i=v_i-v_i^{-1}.\]
Since $u_i=0$, we conclude that $v_i=\pm 1$, and therefore $\Px_i=\pm I$.
It follows that $c_1=\pm c_2$ and so $\Cx_1= \pm  \Cx_2$.
\end{proof}

\subsection{3-sided Regions}  
For 3-sided region of a link diagram, we do not have predetermined edge and crossing matrices, but one can write simplified equations for matrix entries as follows.

\begin{lemma} \label{lemma:3gon}
Consider a 3-sided region with region equation $\Px_1 \Cx_1\Px_2\Cx_2\Px_3\Cx_3=(\pm) I$.  Then $v_i c_i^2 = u_iu_{i+1}v_{i+1}$, where $i=1, 2, 3,$ and subscripts are considered modulo 3. This implies that $c_1^2c_2^2c_3^2=  u_1^2u_2^2u_3^2$ as well.
\end{lemma}

\begin{proof}
The lemma follows from multiplying out the matrices $\Px_1 \Cx_1\Px_2\Cx_2\Px_3\Cx_3$. The  $(2,1)$-entry of this product is 
\[
-(c_1^2v_1 - u_1u_2v_2)c_3v_3/(c_1c_2v_2).
\]

Since this entry must equal to zero, we conclude that $c_1^2v_1 = u_1u_2v_2$. The lemma then follows from considering the relation $\Px_2\Cx_2\Px_3\Cx_3\Px_1 \Cx_1=(\pm) I$ and its other cyclic orderings.

\end{proof}

\section{Cusp Shape}\label{CuspShape}

As a direct consequence of our algorithm, one can determine the cusp shape of a parabolic representation, and its analog for a non-parabolic one. For the longitude, let $\rho(\lambda)=\mathcal{L}$. A longitude $\lambda$ of an $n$-crossing knot consists of concatinated peripheral arcs: \[\beta_{1,R} \beta_{2, R} ... \beta_{2n, R} = \beta_{1,L} \beta_{2, L} ... \beta_{2n, L} =\lambda.\] The product of the corresponding edge matrices gives a formula for the cusp shape (or equivalently, for the length of the knot longitude, when meridian length is fixed to be 1): $\mathcal{L}=\prod_{i=1}^{2n} \Px_{i,R}=\prod_{i=1}^{2n}\Px_{i,L}$, where the $\Px$ matrices are as in Definition \ref{definition:PxandCx}.

Parabolic representations, including discrete and faithful representations, are those where  the $\Mx$ and  $\Px$ matrices are parabolic so that  $m=\pm 1$ and $v_i=\pm 1$.  Therefore, the cusp shape is given by the $(1,2)$-entry of $\mathcal{L}$. Hence for the discrete faithful representation, such an entry is a sum of $u_i$ for all  $\Px_{i,R}$ (or $\Px_{i,L}$).

\section{Example: figure-eight knot}  \label{figure-eight}

As an illustration of our method, we apply Algorithm \ref{algorithm:main} to the figure-eight knot step by step, using Lemma~\ref{lemma:bigon} to simplify computation. We obtain simple equations for the   geometric component of the representation variety.  The equation  govern both the $\SL_2(\C)$ representation variety and the $\PSL_2(\C)$ representation variety. The concrete difference between the two is whether we determine Wirtinger generators by specifying a signed matrix for $\Mx$ or a coset.  We also show how to quickly get  $A$-polynomial, parabolic representations, determine the cusp shape formula for different hyperbolic  structures, and  traceless representations. Equations for the character variety and representation variety for the figure-eight knot are well known. Our results derive this information directly from a diagram and not a presentation for the fundamental group or a triangulation.

\begin{figure}[!ht]
\centering
\includegraphics[scale=0.65]{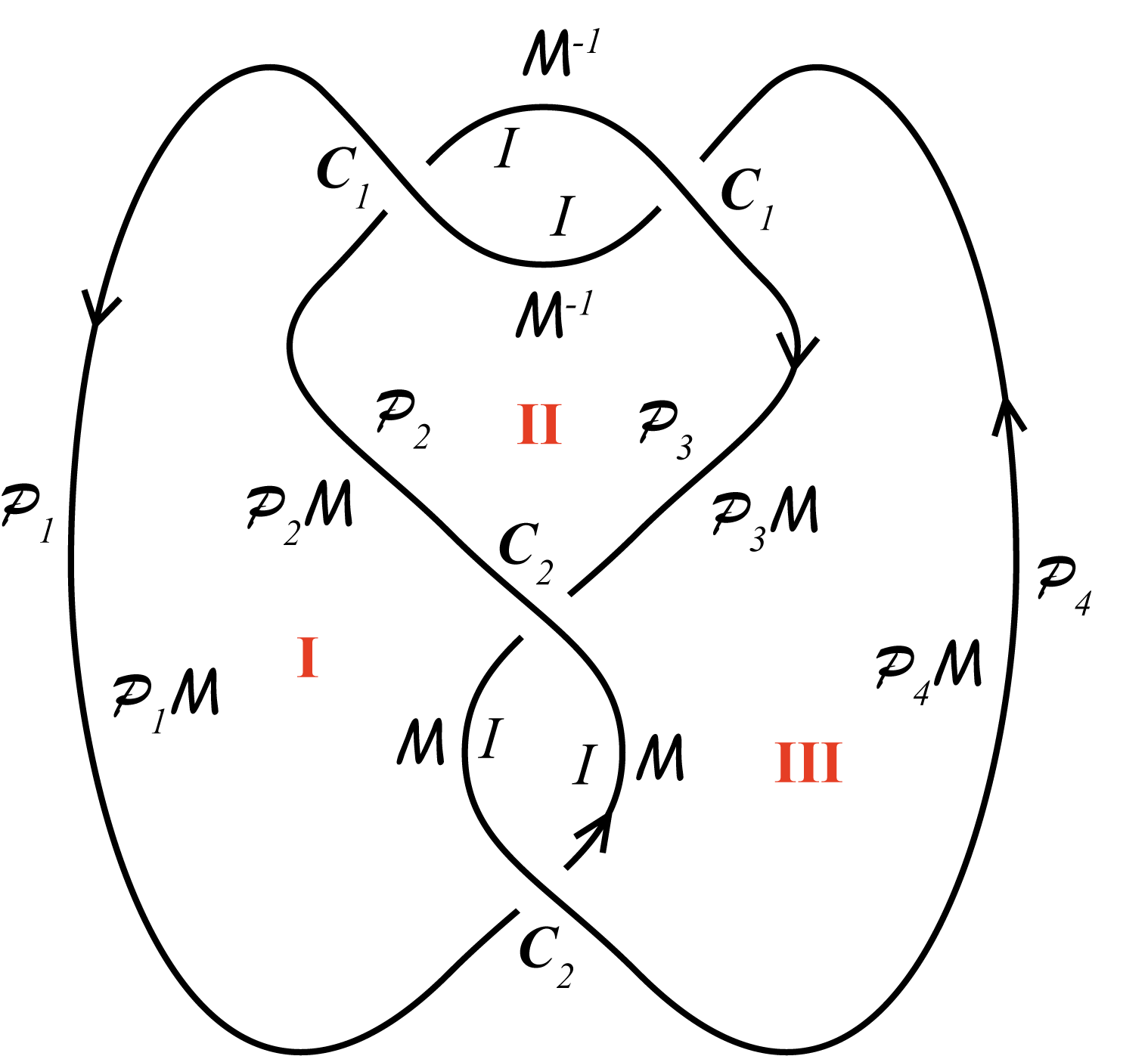}
\caption{Figure-8 knot diagram with edge and crossing matrices that label peripheral and crossing arcs. The regions are labelled by red Roman numerals.}
\label{fig8_labels}
\end{figure}

Choose an orientation for the knot as in Figure \ref{fig8_labels}. Using edge equations, the peripheral labels can all be written in terms of the matrices $\Mx, \Px_1, \Px_2, \Px_3$ and $\Px_4$ as in Figure \ref{fig8_labels}. Additionally, using Lemma~\ref{lemma:bigon} we can write all crossings as $\Cx_1$ or $\Cx_2$ as in Figure~\ref{fig8_labels} which matches the Figure-8 example in \cite{MR3190595}. As mentioned in Remark~\ref{remark:liftingtoSL}, we will choose signs for our $\Px$ and $\C$ matrices and set region equations equal to  $\pm I$ in $\SL_2(\C)$.

The region equations are as follows for regions I, II, III respectively:
\begin{equation}\tag{\text{Region I}}   \Mx^{-1} \Cx_2\Px_2\Mx\Cx_1 \Px_1\Mx\Cx_2 = \pm I \end{equation}  
\begin{equation}\tag{\text{Region II}}  \Px_3 \Cx_2 \Px_2 \Cx_1 \Mx \Cx_1 = \pm I  \end{equation}  
\begin{equation}\tag{\text{Region III}} \Mx^{-1} \Cx_2 \Px_4 \Mx \Cx_1 \Px_3 \Mx \Cx_2 = \pm I  \end{equation} 
For simplicity, we will use $\Wx_i = \mat{0}{-w_i}{1}{0}$ instead of $\Cx_i$ at the cost of a constant multiplier as in Definition \ref{definition:PxandCx}. Therefore $c_i^2=w_i$ for $i=1,2$.
These along with the four commuting equations 
\[
u_i(m^2-1)v_i= (v_i^2-1)m 
\]
for $i=1,2,3,4$ determine the representatives in terms of the parameters $m$, $u_i$, $v_i$, $w_1$, and $w_2$ for $i=1,2,3,4$.  
We will call the left hand side matrices of these equations $\Rx_I$, $\Rx_{II}$ and $\Rx_{III}$ respectively so that this algebraic set is determined by the commuting equations and the equations $\Rx[1,2]=\Rx[2,1]=0$ and $\Rx[1,1]=\Rx[2,2]$  for each of the $\Rx$ as above where $[a,b]$ indicates the appropriate matrix entries.  These equations hold even with the use of the $\Wx_i$ matrices.

From the region equations, using resultants, we obtain  $\Px_1=\Px_3$ and $\Px_2=\Px_4$. Specifically, 
\begin{equation}\tag{$*$} 
v_1=v_3, \  u_1=u_3, \ v_2=v_4, \ u_2=u_4,  \ w_1w_2=1.
 \end{equation}
With these reductions, one also immediately obtains the following linear relations:
\begin{equation}\tag{$**$} 
w_1 = u_1^2 u_2^2, \quad 
v_1 = m u_1 u_2^2, \quad 
v_2 = m u_2 w_1^{-1} = mu_1^{-2}u_2^{-1}.
 \end{equation}
Because equations  $(*)$ and $(**)$  are linear, we can remove the variables  $v_1$,  $v_2$,  $v_3$, $u_3$, $v_4$, $u_4$ $w_1$, and $w_2$  from our equations to obtain an isomorphic algebraic set defining our solutions. 
With these the variables $u_1, u_2$, and $m$ are related as follows:
\begin{equation}\tag{$*\!*\!*$}  m^4 + m^2u_1^2u_2^2 + u_1=0  \end{equation}
\begin{equation}\tag{$*\!*\!**$}  u_1^6 u_2^6 - u_1^6 u_2^4 - u_1^4u_2^6 + 2u_1^2u_2^2 - 1=0. \end{equation}
With $(*), (**), (*\!*\!*), (*\!*\!**)$, a given $u_2$ determines a finite number of $u_1$ and $m$ values, and all other values are completely determined by these parameters.

{\em Parabolic Representations}:     Let $m=1$. Then $v_1=v_2=v_3=v_4=1$ and $u_1=u_2=u_3=u_4=w_1$. For each $u_i$, we have $u_i^2+u_i+1=0$, so $u_i$ is a primitive third root of unity. Moreover, $w_2=w_1^{-1}=w_1^2$. This recovers Example 6.1 from \cite{MR3190595}, giving all parabolic representations that lie on the canonical component. 

{\em Cusp Shape}: The above relations for parabolic representations and  the formula in Section \ref{CuspShape} allow us to compute  the cusp shape. The cusp shape is the (1, 2)-entry of the matrix product $\Px_1\Mx\Px_2\Px_3\Mx\Px_4$. We can simplify this using the fact, mentioned above, that $\Px_3=\Px_1$ and $\Px_4=\Px_2$, together with the fact that all of these peripheral matrices commute. Hence we can write the longitude as 
\begin{equation}\tag{$\tilde{o}$}  \Lx =\pm (\Px_1 \Px_2 \Mx)^2.\end{equation} 
This gives the cusp shape of $4u_1+2$. For the complete hyperbolic structure, it is $2\sqrt{3}i$.

{\em Traceless Representations}:   
Traceless Representations of knot groups are representations where the meridian is sent to a matrix of trace zero. These representations often showcase connections to other invariants and related manifolds. (See, for example \cite{MR2983076}, \cite{MR3158776}, and \cite{MR3835325}.) 
For the figure-8 knot complement, when $\Mx$ is traceless (so $m^2+1=0$), we can compute them as follows. 

From $(**)$,  $1-u_1^2u_2^2+u_1^2=0$, and from $(*\!*\!*)$, $u_1^6 u_2^6 - u_1^6 u_2^4 - u_1^4u_2^6 + 2u_1^2u_2^2 - 1=0$.  Taking resultants, we see that $u_1^4 - u_1^2 - 1=0.$ 
We conclude that the following hold, where the $\pm$ are chosen consistently (the choice of all the upper signs or all the lower signs):
\[
u_1^2=u_2^2 = \tfrac12(1\pm \sqrt{5}), v_1^2= -u_1^6=-(2\pm \sqrt{5}), v_2^2 = -u_2^{-6} = 2\mp \sqrt{5}.
\]
We also have that 
\[
w_1= \tfrac12(3\pm \sqrt{5}), w_2 = w_1^{-1} = \tfrac12(3\mp \sqrt{5}).
\]

{\em $A$-polynomial}:
The $A$-polynomial was defined in \cite{MR1288467}, and there the $A$-polynomial for the figure-8 knot was computed as 
\[ L^2M^4+L(-M^8+M^6+2M^4+M^2-1)+M^4.\]
This is well-defined up to multiplication by a constant and powers of $L$ and $M$. 

From ($\tilde{o}$), the longitude is  $\Lx = (\Px_1 \Px_2 \Mx)^2$, and this is well-defined up to sign (as mentioned in Remark~\ref{remark:liftingtoSL}). 
We can set the $[1,1]$ entry of $\Lx$  to $L$ and use $M$ to denote the $[1,1]$ entry of $\Mx$ (so that $M=m$). In this way we will write the $A$-polynomial as a polynomial in the variables $L$ and $M$.   Upon choosing ``-" sign in ($\tilde{o}$): $\Lx = -(\Px_1 \Px_2 \Mx)^2$, and taking resultants to eliminate all variables other than $L$ and $M$ in this equation, as well as in $(*\!*\!*)$, and $(*\!*\!**)$, we get the same $A$-polynomial as above.

{\em Character Variety}: We can also compute the $\SL_2(\C)$ character variety upon letting 
$x=m+m^{-1}$ and letting $z= \tr (\Cx_2 \Mx \Cx_2^{-1}) = 2-c_1^{-2}=2-w_2^{-1}=2-w_1$. We have $w_1=u_2^2u_2^2$  from $(**)$, so that $z=2-u_2^2u_2^2$ with the two defining equations ($*\!*\!*$) and ($*\!*\!**$) above. Taking resultants with $m^2-mx+1=0$, we have: 
\[
x^2z-2x^2-z^2+z+1 = 0.
\]
This is a defining equation for the (non-abelian) portion of the character variety for the figure-8 knot complement. It can be rewritten as 
\[
x^2(z-2) = z^2-z-1.
\]
 Upon a change of variables $y=x(z-2)$, this is birational to
 \[
 x^2 = (z^2-z-1)(z-2) = z^3-3z^2+z+2.
 \]
Note that substituting $y-2= r-1$ we recover $x^2(r-1)= r^2+r-1$, which is the form that one gets from the standard two bridge presentation $\langle a,b | aw=wb\rangle$ with $w=ba^{-1}b^{-1}a$ taking $x=\chi_{\rho}(a)=\chi_{\rho}(b)$ and $r=\chi_{\rho}(ab^{-1})$.

\section{3-braids}\label{3-braids}

\begin{example} Knots from the infinite family of closed alternating braids with the braid word
$(\sigma_1(\sigma_2)^{-1})^n$, $n>2$.

Note that when $n=3k$, where $k$ is natural number, this is a 3-component link (for example, Borromean link for $n=3$), and otherwise it is a knot (for example, for $n=4$, Turk's Head Knot). The procedure below is for a knot complement in $S^3$, i.e. for the case where $n\neq 3k$, for any natural $k$.

According to Step 1 of Algorithm~\ref{algorithm:main}, we first assign meridian matrix to a meridian, and edge and crossing matrices to the reduced alternating diagram as in Figure \ref{braid}.  Here  $\Px_i, i=1, 2, \dots, 8,$ are edge matrices, and $\Cx_1, \Cx_2$ are crossing matrices  as in  Definition \ref{definition:PxandCx}. 

 \begin{figure}[!ht]
  \includegraphics[scale=0.4]{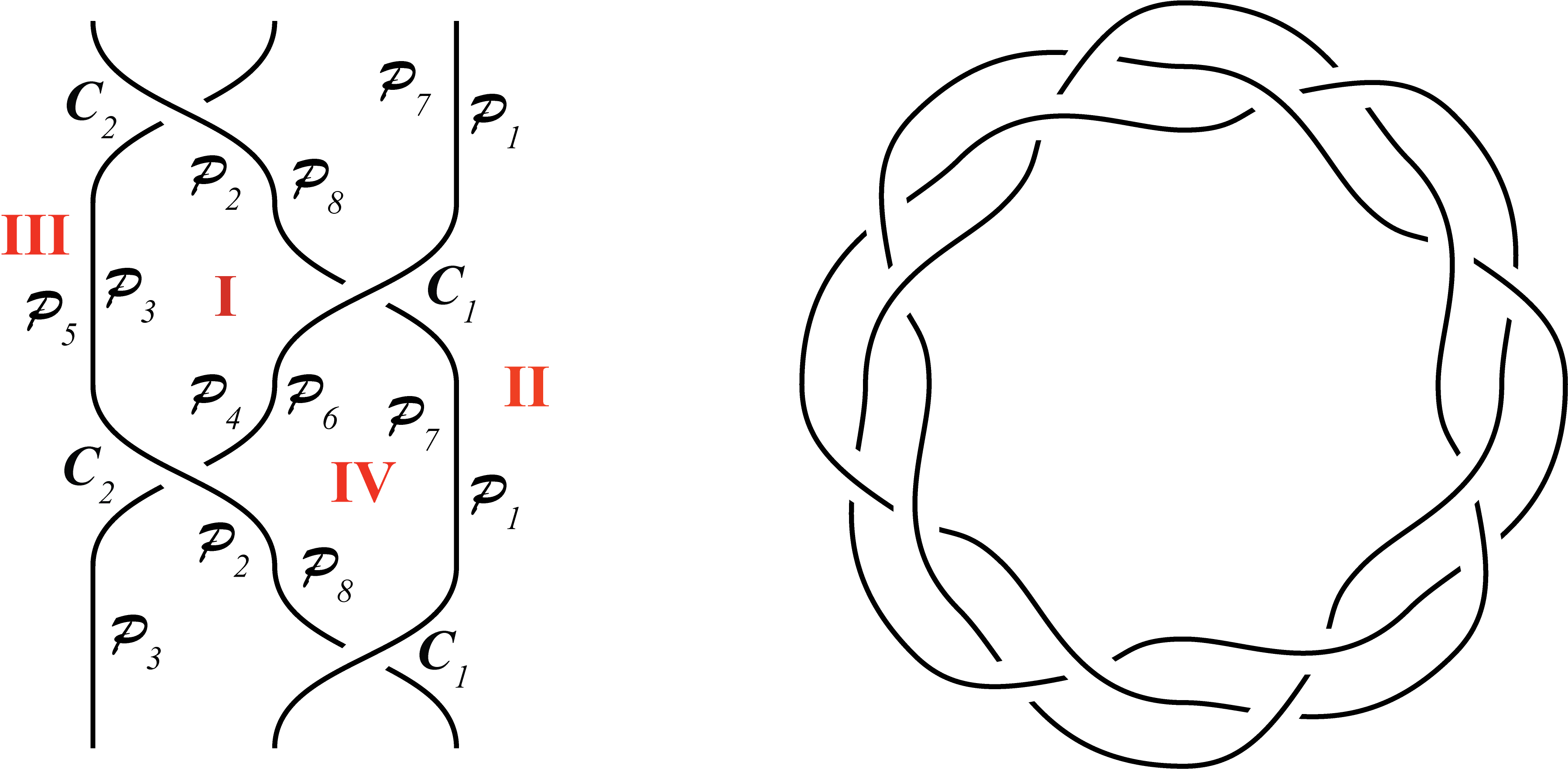}  
  \caption{On the right, a closed alternating braid from the family $(\sigma_1(\sigma_2)^{-1})^n$, $n>2$. On the left, a fragment of such diagram labelled by edge and crossing matrices (in black) and with numbered regions (with red Roman numerals).}
\label{braid}

  \end{figure}

We will use a symmetric labeling of the knot which corresponds to the central  rotational symmetry of the knot.  It is known  \cite{MR387430, MR2796637}   by looking at orbifold quotients that infinitely many representations, including discrete faithful representations and infinitely many representations corresponding to Dehn fillings satisfy this symmetry.  These representations can  be characterized as those which factor through the fundamental group of the orbifold obtained by taking the quotient by the action of the symmetry.

As a result of this simplification,  all of  the crossing matrices on one side of the braid are $\Cx_1$, on  the other side are $\Cx_2$. One can see a similar repeating pattern for edge matrices in Figure \ref{braid}.  The matrices are $\Px_i=\left(\begin{array}{ccc}
v_i &  u_i  \\
0 & v_i^{-1} \end{array} \right)$ as in Definition~\ref{definition:PxandCx}, $\Cx_i=\left(\begin{array}{ccc}
0 & c_i  \\
-c_i^{-1} & 0 \end{array} \right)$, and the meridian $\Mx=\left(\begin{array}{ccc}
m & 1  \\
0 & m^{-1} \end{array} \right)$.

 We assume that all of the labels $v_i, i=1,2,3,4,$ and $m$ are not 0. With this and the numbering of the diagram regions as on the figure (in red Roman numerals), we have the following, where the braid is oriented upwards.

\begin{equation}\tag{\text{Region I}} \Px_3\Cx_2\Px_2^{-1} \Cx_1\Px_4^{-1} \Cx_2=\pm I\end{equation}
\begin{equation}\tag{\text{Region II}}  (\Px_1\Cx_1)^n=\pm I \end{equation}
\begin{equation}\tag{\text{Region III}} (\Px_5\Cx_2)^n=\pm I   \end{equation}
\begin{equation}\tag{\text{Region IV}}  \Px_6\Cx_1\Px_7^{-1}\Cx_1\Px_8\Cx_2 =\pm I  \end{equation}

\noindent Here $I$ is the identity matrix. 

We will use Lemma~\ref{lemma:3gon}  instead of the  Region I and IV equations.

\subsection{Regions II and III}

Any equation of the form $(\Px_i\Cx_j)^n=\pm I$ is 
\[
\mat{-u_ic_j^{-1}}{v_ic_j}{-v_i^{-1}c_j^{-1}}{0}^n=\pm I.
\]
We conclude that up to sign the trace, $(\pm) u_ic_j^{-1}$ equals $\zeta + \zeta^{-1}$ for $\zeta$ some $n$-th root of unity and so 
\begin{equation}\tag{$\dagger$} -u_ic_j^{-1} = 2 \cos (\pi k/n). \end{equation}
We have used the crossing matrices in the form of $\Cx_j$ here to underscore this arithmetic, since for $\Cx_j$, it is straightforward to write the condition that the $n$th power of a determinant one matrix is $\pm I$ in terms of the trace.
 Below, we will use $w_j=c_j^2$ instead, where $w_j$ corresponds to the  crossing matrix $\Wx_j$ (Definition \ref{definition:PxandCx}), to make the calculations more streamlined.

We can also express the condition that the $n$th power of the matrix is $\pm I$ recursively.  By the Cayley-Hamilton theorem, for $M\in \SL_2(\C)$ we can write $M^n=f_n(\tr M) M-f_{n-1}(\tr M) I$, where $f_{\ell}$ is defined recursively for both positive and negative $\ell$ by $f_0(x)=0$, $f_1(x)=1$ and $f_{\ell+1}(x) +f_{\ell-1}(x) = x f_{\ell}(x)$. 
Therefore, with $M=\Px_i\Cx_j$,
\[
M^n = f_n(-u_ic_j^{-1}) M -  f_{n-1}(-u_ic_j^{-1}) I, 
\] 
and the traces of the left and right sides of the above equation are related by
\[
 \pm ( m^n+m^{-n})= -u_ic_j^{-1}f_n(-u_ic_j^{-1}) -2  f_{n-1}(-u_ic_j^{-1}). 
\]

With the commuting equations, $u_i  (m-m^{-1})=(v_i-v_i^{-1})$, we can write $(\dagger)$ as follows after squaring both sides of the equality:
\[
4 w_j (m-m^{-1})^2  \cos^2 (\pi k/n) = (v_i-v_i^{-1})^2 
\]
for some $k\in \Z$. 
Region II and III equations are $(\Px_1\Cx_1)^n\pm I$ and $(\Px_5\Cx_2)=\pm I$ so that
\begin{equation}\tag{$\dagger \dagger$}
\begin{aligned}
4w_1(m-m^{-1})^2 \cos^2 (\pi k_1/n)& =  (v_1-v_1^{-1})^2,  \\   4 w_2(m-m^{-1})^2 \cos^2 (\pi k_2/n) & = (v_5-v_5^{-1})^2.
\end{aligned}
\end{equation}
Note that if the representation is parabolic, then ($\dagger \dagger$) is the trivial equation since then certain matrix entries must be 1: $m=v_1=v_5=\pm 1$.   As a result, we will compute the parabolic representations separately. In the parabolic case, the identity ($\dagger$) gives us 
\begin{equation}\tag{$\dagger \dagger \dagger$}
 u_1^2 = 4 w_1  \cos^2 (\pi k_1/n), \quad  u_5^2 = 4 w_2  \cos^2 (\pi k_2/n).
\end{equation}
 
\subsection{Regions I and IV}

Using Lemma~\ref{lemma:3gon} we obtain the following equations from Region I
\[
w_2 = -u_2v_2^{-1}u_3v_3^{-1}, \ w_1  = u_2v_2u_4 v_4^{-1}, \  w_2   = - u_4v_4u_3v_3
\]
and for Region IV we have
\[
w_1   = -u_6v_6^{-1} u_7v_7^{-1},  \ w_1   =  -u_7v_7u_8v_8,  \ w_2   =  u_8v_8^{-1}u_6v_6.
\]

Assume that the representation is not parabolic.
The commuting equations imply that $u_iv_i = (v_i^2-1)(m-m^{-1})^{-1}$ and  $u_iv_i^{-1} = (1-v_i^{-2})(m-m^{-1})^{-1}$ and upon replacing these above in the equations for $w_1$ and $w_2$ and collecting the expressions in terms of $w_1$ or $w_2$ we have 
\begin{align*}
w_1 (m-m^{-1})^2 & = (v_2^2-1)(1-v_4^{-2}) = -(1-v_6^{-2})(1-v_7^{-2}) = - (v_7^2-1)(v_8^2-1) \\
w_2 (m-m^{-1})^2 & =  -(1-v_2^{-2}) (1-v_3^{-2}) = -(v_4^2-1)(v_3^2-1)= (1-v_8^{-2})(v_6^2-1).
\end{align*}

Here are edge equations for our knot:
\[
\Px_6= \pm \Mx \Px_4, \ \Px_7 = \pm  \Mx \Px_1, \ \Px_8 = \pm  \Mx \Px_2, \ \Px_5 = \pm  \Mx \Px_3.
\]
They imply the following relations for the $m$ and $v_i$ variables:
\[
v_ 6  = \pm mv_4,
v_ 7  = \pm mv_1,
v_ 8  = \pm mv_2,
v_ 5  = \pm mv_3.
\]
With these we replace $v_3, v_8, v_6$, and $v_7$ in the $w_1 (m-m^{-1})^2$ and $w_2 (m-m^{-1})^2$ equations above and get the following.

\begin{equation}\tag{$*$}
\begin{aligned}
w_1 (m-m^{-1})^2 &  = (v_2^2-1)(1-v_4^{-2}) = -(1-m^{-2}v_4^{-2})(1-m^{-2}v_1^{-2})  \\
&  = - (m^2v_1^2-1)(m^2v_2^2-1) 
\end{aligned}
\end{equation}
\begin{equation}\tag{$**$}
\begin{aligned}
w_2 (m-m^{-1})^2  & =  -(1-v_2^{-2}) (1-m^2v_5^{-2}) = -(v_4^2-1)(m^{-2}v_5^2-1) \\
& = (1-m^{-2}v_2^{-2})(m^2v_4^2-1).
 \end{aligned}
\end{equation}

Equation $(**)$  implies that $-(1-v_2^{-2}) (1-m^2v_5^{-2}) +(v_4^2-1)(m^{-2}v_5^2-1)=0$, which reduces to 
\[ v_4^2 -1 = m^2v_5^{-2}(1-v_2^{-2}),\]
 if $v_5^2\neq m^2$.

In the equation directly above, use the right side to express $v_4$ in terms of the other variables. Now there are four equal expressions in $(*)$. Take the difference of the second and fourth expressions and replace $v_4$ with the expression for $v_4$ that we obtained.  Call the numerator $A$ so that $A=0$. Similarly, there are  four equal expressions in $(**)$.   Take the difference of the second and fourth expressions and replace $v_4$ too.  Call the numerator $B$ so that $A=B=0$. Hence $v_2^2B+A=0$, and we obtain:
 \[
v_2^2 = \frac{ m^4v_1^2 - m^2v_1^2v_5^2  - m^4  + v_1^2v_5^2  + m^2  - v_5^2 }{ m^2(m^2v_1^2 - m^2 - v_5^2 + 1)}.
 \]
It follows that
\begin{align*}
m^6v_1^4v_5^2 & - m^4v_1^4v_5^4 - m^6v_1^4 - m^6v_1^2v_5^2 + m^2v_1^4v_5^4 + m^6v_1^2 + 2m^4v_1^2v_5^2  & \\
& - 2m^2v_1^2v_5^2  - v_1^2v_5^4 - m^4 + v_1^2v_5^2 + v_5^4 + m^2 - v_5^2=0. 
\end{align*}

\subsection{Defining Equations}

The Region II and III equations with the substitutions from $(*)$ and $(**)$ give us that for any integers $k_1, k_2$,
\begin{align*}
- 4(m^2v_1^2-1)(m^2v_2^2-1)  \cos^2 (\pi k_1/n)& =  (v_1-v_1^{-1})^2  \\   
 -4(v_4^2-1)(m^{-2}v_5^2-1) \cos^2 (\pi k_2/n) & = (v_5-v_5^{-1})^2.
\end{align*}

From the previous section using the $v_4^2-1$ and $v_2^2$ equations,  we can write 
\[
v_4^2-1 = - \frac{m^2(m^2 - 1)(v_1^2 - 1)}{ m^4v_1^2  - m^2 v_1^2v_5^2  - m^4 + v_1^2v_5^2  + m^2 - v_5^2}
\]
and 
\[
m^2v_2^2-1= \frac{(m^2v_1^2- v_1^2v_5^2 - m^2 + 1)(m^2- 1)}{ m^2v_1^2  - m^2  - v_5^2 + 1}
\]
These equations determine the non-parabolic representations.

\subsection{Parabolic Representations}

Recall that ($\dagger \dagger \dagger$) are 
\[ u_1^2= 4 w_1 \cos^2 ( \pi k_1/n), \quad u_5^2  = 4 w_2  \cos^2 ( \pi k_2/n).\] 
Let $q_i =  2\cos ( \pi k_i/n)$. 
The commuting equations imply that all $v_i=\pm 1$.

Lemma~\ref{lemma:3gon} gives the following equations from Region I and IV
\[
w_1  = u_2 u_4   = -u_6  u_7  =  -u_7u_8,  \  w_2 = -u_2 u_3  = - u_4 u_3  =  u_8 u_6. 
\]
Therefore $u_6=u_8$ and $u_2=u_4$. The edge equations imply that 
\[
 u_7=u_1+1, u_8=u_2+1, u_5=u_3+1.
\]
We conclude that  parabolic representations can be described by $u_1, u_2, u_3$, $w_1$, and $w_2$.
We can rewrite the equations above as
\[
q_1^2w_1=q_1^2u_2^2=-q_1^2(u_1+1)(u_2+1)=u_1^2, \quad q_2^2w_2=q_2^2(u_2+1)^2=-q_2^2u_2u_3=(u_3+1)^2.
\]
Therefore $u_1 = \pm q_1 u_2$ and $u_3+1 = \pm  q_2 (u_2+1)$. And the equations $-(u_1+1)(u_2+1)=u_2^2,-u_2u_3=(u_2+1)^2$ reduce to 
\[(1 \pm q_1) u_2( u_2 +1)+1 =0, \quad (1 \pm q_2)u_2( u_2 +1)+1 =0. \]
Hence $q_1=\pm q_2$.  With $w_1=u_2^2$, $w_2=(u_2+1)^2$, and $u_3=-\tfrac1{u_2}(u_2+1)^2$, we have
\[
(q_1+1)u_2^2+(q_1+1)u_2+1=0.
\]
This recovers Example 6.2 from \cite{MR3190595}.

\end{example}

\bibliographystyle{amsplain}
\bibliography{knots}

\providecommand{\bysame}{\leavevmode\hbox to3em{\hrulefill}\thinspace}
\providecommand{\MR}{\relax\ifhmode\unskip\space\fi MR }
\providecommand{\MRhref}[2]{%
  \href{http://www.ams.org/mathscinet-getitem?mr=#1}{#2}
}
\providecommand{\href}[2]{#2}
\begin{thebibliography}{10}

\bibitem{MR1876890}
Colin~C. Adams, \emph{Waist size for cusps in hyperbolic 3-manifolds}, Topology
  \textbf{41} (2002), no.~2, 257--270. \MR{1876890}

\bibitem{MR1670053}
S.~Boyer and X.~Zhang, \emph{On {C}uller-{S}halen seminorms and {D}ehn
  filling}, Ann. of Math. (2) \textbf{148} (1998), no.~3, 737--801.
  \MR{1670053}

\bibitem{MR1909249}
Steven Boyer and Xingru Zhang, \emph{A proof of the finite filling conjecture},
  J. Differential Geom. \textbf{59} (2001), no.~1, 87--176. \MR{1909249}

\bibitem{MR3835325}
Haimiao Chen, \emph{Trace-free {${\rm SL}(2,\Bbb C)$}-representations of
  {M}ontesinos links}, J. Knot Theory Ramifications \textbf{27} (2018), no.~8,
  1850050, 10. \MR{3835325}

\bibitem{MR1288467}
D.~Cooper, M.~Culler, H.~Gillet, D.~D. Long, and P.~B. Shalen, \emph{Plane
  curves associated to character varieties of {$3$}-manifolds}, Invent. Math.
  \textbf{118} (1994), no.~1, 47--84. \MR{1288467}

\bibitem{MR825087}
Marc Culler, \emph{Lifting representations to covering groups}, Adv. in Math.
  \textbf{59} (1986), no.~1, 64--70. \MR{825087}

\bibitem{MR881270}
Marc Culler, C.~McA. Gordon, J.~Luecke, and Peter~B. Shalen, \emph{Dehn surgery
  on knots}, Ann. of Math. (2) \textbf{125} (1987), no.~2, 237--300.
  \MR{881270}

\bibitem{MR683804}
Marc Culler and Peter~B. Shalen, \emph{Varieties of group representations and
  splittings of {$3$}-manifolds}, Ann. of Math. (2) \textbf{117} (1983), no.~1,
  109--146. \MR{683804}

\bibitem{Handbook}
Edited by R.~J. Daverman and R.~B Sher, \emph{Handbook of geometric topology},
  North-Holland, Amsterdam (2002), 1133 pp.

\bibitem{Nathan}
Nathan Dunfield, \emph{personal correspondence},  (2023).

\bibitem{MR1695208}
Nathan~M. Dunfield, \emph{Cyclic surgery, degrees of maps of character curves,
  and volume rigidity for hyperbolic manifolds}, Invent. Math. \textbf{136}
  (1999), no.~3, 623--657. \MR{1695208}

\bibitem{Rocky}
Rochy Flint, \emph{Intercusp geodesics and cusp shapes of fully augmented
  links}, preprint (2018).

\bibitem{Quasi}
David Futer, Efstratia Kalfagianni, and Jessica~S. Purcell, \emph{Quasifuchsian
  state surfaces}, Trans. Amer. Math. Soc. \textbf{366} (2014), no.~8,
  4323--4343.

\bibitem{MR3325748}
Stavros Garoufalidis, Matthias Goerner, and Christian~K. Zickert, \emph{Gluing
  equations for {$\rm{PGL}(n,\Bbb{C})$}-representations of 3-manifolds},
  Algebr. Geom. Topol. \textbf{15} (2015), no.~1, 565--622. \MR{3325748}

\bibitem{MR1248117}
F.~Gonz\'alez-Acu\~na and Jos\'e Mar\'\i~a Montesinos-Amilibia, \emph{On the
  character variety of group representations in {${\rm SL}(2,{\bf C})$} and
  {${\rm PSL}(2,{\bf C})$}}, Math. Z. \textbf{214} (1993), no.~4, 627--652.
  \MR{1248117}

\bibitem{MR3158776}
Matthew Hedden, Christopher~M. Herald, and Paul Kirk, \emph{The pillowcase and
  perturbations of traceless representations of knot groups}, Geom. Topol.
  \textbf{18} (2014), no.~1, 211--287. \MR{3158776}

\bibitem{Software}
Bae Jaeyun, Mark Bell, Dale Koenig, Alex Lowen, and Anastasiia Tsvietkova,
  \emph{Geometric structures from diagrams},
  www.github.com/Nastia-Tsvietkova/Geometric-Structures.

\bibitem{MR3876301}
Hyuk Kim, Seonhwa Kim, and Seokbeom Yoon, \emph{Octahedral developing of knot
  complement {I}: {P}seudo-hyperbolic structure}, Geom. Dedicata \textbf{197}
  (2018), 123--172. \MR{3876301}

\bibitem{TorusTT}
Alice Kwon, Byungdo Park, and Ying~Hong Tham, \emph{Generalization of the
  {Thistlethwaite}--{Tsvietkova} method}, preprint (2023).

\bibitem{MR1739217}
D.~D. Long and A.~W. Reid, \emph{Commensurability and the character variety},
  Math. Res. Lett. \textbf{6} (1999), no.~5-6, 581--591. \MR{1739217}

\bibitem{MR2796637}
\bysame, \emph{Fields of definition of canonical curves}, Interactions between
  hyperbolic geometry, quantum topology and number theory, Contemp. Math., vol.
  541, Amer. Math. Soc., Providence, RI, 2011, pp.~247--257. \MR{2796637}

\bibitem{MR2827003}
Melissa~L. Macasieb, Kathleen~L. Petersen, and Ronald~M. van Luijk, \emph{On
  character varieties of two-bridge knot groups}, Proc. Lond. Math. Soc. (3)
  \textbf{103} (2011), no.~3, 473--507. \MR{2827003}

\bibitem{MR1937957}
Colin Maclachlan and Alan~W. Reid, \emph{The arithmetic of hyperbolic
  3-manifolds}, Graduate Texts in Mathematics, vol. 219, Springer-Verlag, New
  York, 2003. \MR{1937957}

\bibitem{MR387430}
Wilhelm Magnus, \emph{Two generator subgroups of {${\rm PSL}$} {$(2,\,C)$}},
  Nachr. Akad. Wiss. G\"ottingen Math.-Phys. Kl. II (1975), no.~7, 81--94.
  \MR{387430}

\bibitem{MR0222880}
Wilhelm Magnus and Ada Peluso, \emph{On knot groups}, Comm. Pure Appl. Math.
  \textbf{20} (1967), 749--770. \MR{0222880}

\bibitem{MR721450}
W.~Menasco, \emph{Closed incompressible surfaces in alternating knot and link
  complements}, Topology \textbf{23} (1984), no.~1, 37--44. \MR{721450
  (86b:57004)}

\bibitem{MR0236383}
G.~D. Mostow, \emph{Quasi-conformal mappings in {$n$}-space and the rigidity of
  hyperbolic space forms}, Inst. Hautes \'{E}tudes Sci. Publ. Math. (1968),
  no.~34, 53--104. \MR{0236383}

\bibitem{MR2983076}
Fumikazu Nagasato and Yoshikazu Yamaguchi, \emph{On the geometry of the slice
  of trace-free {$SL_2(\Bbb C)$}-characters of a knot group}, Math. Ann.
  \textbf{354} (2012), no.~3, 967--1002. \MR{2983076}

\bibitem{MR3450771}
Kathleen~L. Petersen and Anh~T. Tran, \emph{Character varieties of double twist
  links}, Algebr. Geom. Topol. \textbf{15} (2015), no.~6, 3569--3598.
  \MR{3450771}

\bibitem{MR0385005}
Gopal Prasad, \emph{Strong rigidity of {${\bf Q}$}-rank {$1$} lattices},
  Invent. Math. \textbf{21} (1973), 255--286. \MR{0385005}

\bibitem{purcell}
Jessica~S. Purcell, \emph{Hyperbolic knot theory}, vol. 209, AMS Graduate
  Studies in Mathematics, 2020.

\bibitem{MR0300267}
Robert Riley, \emph{Parabolic representations of knot groups. {I}}, Proc.
  London Math. Soc. (3) \textbf{24} (1972), 217--242. \MR{0300267}

\bibitem{MR0413078}
\bysame, \emph{Parabolic representations of knot groups. {II}}, Proc. London
  Math. Soc. (3) \textbf{31} (1975), no.~4, 495--512. \MR{0413078}

\bibitem{MR745421}
\bysame, \emph{Nonabelian representations of {$2$}-bridge knot groups}, Quart.
  J. Math. Oxford Ser. (2) \textbf{35} (1984), no.~138, 191--208. \MR{745421}

\bibitem{Segerman}
Henry Segerman, \emph{A generalisation of the deformation variety}, Algebr.
  Geom. Topol. \textbf{12} (2012), no.~4, 2179--2244. \MR{3020204}

\bibitem{MR3190595}
Morwen Thistlethwaite and Anastasiia Tsvietkova, \emph{An alternative approach
  to hyperbolic structures on link complements}, Algebr. Geom. Topol.
  \textbf{14} (2014), no.~3, 1307--1337. \MR{3190595}

\bibitem{thurston}
William~P. Thurston, \emph{The geometry and topology of three-manifolds. {V}ol.
  {IV}}, American Mathematical Society, Providence, RI, [2022] \copyright 2022,
  Edited and with a preface by Steven P. Kerckhoff and a chapter by J. W.
  Milnor. \MR{4554426}

\bibitem{MR3073918}
Anh~T. Tran, \emph{The universal character ring of some families of one-relator
  groups}, Algebr. Geom. Topol. \textbf{13} (2013), no.~4, 2317--2333.
  \MR{3073918}

\bibitem{thesis}
Anastasiia Tsvietkova, \emph{Hyperbolic {S}tructures from {L}ink {D}iagrams},
  ProQuest LLC, Ann Arbor, MI, 2012, Thesis (Ph.D.)--The University of
  Tennessee. \MR{3121975}

\bibitem{SnapPea}
J.~R. Weeks, \emph{Snappea: a computer program for creating and studying
  hyperbolic 3{ manifolds}, journal =
  {http://thames.northnet.org/weeks/index/SnapPea.html},}.

\end{thebibliography}

\end{document}